\documentclass[notitlepage,11pt,reqno]{amsart}
\usepackage[foot]{amsaddr}
\usepackage{amssymb,nicefrac,bm,upgreek,mathtools,verbatim}
\usepackage[final]{hyperref}
\usepackage[mathscr]{eucal}
\usepackage{dsfont}
\usepackage[normalem]{ulem}
\usepackage{amsopn}

\usepackage[margin=1in]{geometry}
\allowdisplaybreaks
\newcommand{\stkout}[1]{\ifmmode\text{\sout{\ensuremath{#1}}}\else\sout{#1}\fi}

\newtheorem{theorem}{Theorem}[section]

\theoremstyle{definition}

\theoremstyle{remark}
\newtheorem{remark}{Remark}[section]
\numberwithin{theorem}{section}
\numberwithin{equation}{section}

\hypersetup{
  colorlinks=true,
  citecolor=mblue,
  linkcolor=mblue,
  urlcolor = blue,
  anchorcolor = blue,
  frenchlinks=false,
  pdfborder={0 0 0},
  naturalnames=false,
  hypertexnames=false,
  breaklinks}
\usepackage[capitalize,nameinlink]{cleveref}
\usepackage[abbrev,msc-links,nobysame,citation-order]{amsrefs} 

\crefname{section}{Section}{Sections}
\crefname{subsection}{Section}{Sections}
\crefname{condition}{Condition}{Conditions}
\crefname{hypothesis}{Hypothesis}{Conditions}
\crefname{assumption}{Assumption}{Assumptions}
\crefname{lemma}{Lemma}{Lemmas}
\crefname{fact}{Fact}{Facts}

\Crefname{figure}{Figure}{Figures}

\crefformat{equation}{\textup{#2(#1)#3}}
\crefrangeformat{equation}{\textup{#3(#1)#4--#5(#2)#6}}
\crefmultiformat{equation}{\textup{#2(#1)#3}}{ and \textup{#2(#1)#3}}
{, \textup{#2(#1)#3}}{, and \textup{#2(#1)#3}}
\crefrangemultiformat{equation}{\textup{#3(#1)#4--#5(#2)#6}}%
{ and \textup{#3(#1)#4--#5(#2)#6}}{, \textup{#3(#1)#4--#5(#2)#6}}%
{, and \textup{#3(#1)#4--#5(#2)#6}}

\Crefformat{equation}{#2Equation~\textup{(#1)}#3}
\Crefrangeformat{equation}{Equations~\textup{#3(#1)#4--#5(#2)#6}}
\Crefmultiformat{equation}{Equations~\textup{#2(#1)#3}}{ and \textup{#2(#1)#3}}
{, \textup{#2(#1)#3}}{, and \textup{#2(#1)#3}}
\Crefrangemultiformat{equation}{Equations~\textup{#3(#1)#4--#5(#2)#6}}%
{ and \textup{#3(#1)#4--#5(#2)#6}}{, \textup{#3(#1)#4--#5(#2)#6}}%
{, and \textup{#3(#1)#4--#5(#2)#6}}

\crefdefaultlabelformat{#2\textup{#1}#3}
\newcommand{\vertiii}[1]{{\left\vert\kern-0.25ex\left\vert\kern-0.25ex\left\vert #1 
    \right\vert\kern-0.25ex\right\vert\kern-0.25ex\right\vert}}

\newcommand{\Uadm}{\mathfrak U}

\newcommand{\Act}{\mathbb{U}}
\newcommand{\Usm}{\mathfrak U_{\mathsf{sm}}}
\newcommand{\Udsm}{\mathfrak U_{\mathsf{dsm}}}

\newcommand{\Um}{\mathfrak U_{\mathsf{m}}}

\newcommand{\pV}{\mathrm{V}} 
\newcommand{\pv}{\mathrm{v}} 

\newcommand{\fB}{{\mathfrak{B}}}  
\newcommand{\cB}{{\mathcal{B}}}  
\newcommand{\sB}{{\mathscr{B}}}  
\newcommand{\cC}{{\mathcal{C}}}   
\newcommand{\sD}{{\mathscr{D}}}   
\newcommand{\sE}{{\mathscr{E}}} 
\newcommand{\sF}{{\mathfrak{F}}}   
\newcommand{\cJ}{{\mathcal{J}}}  
  
\newcommand{\sK}{{\mathscr{K}}}  
\newcommand{\sL}{{\mathscr{L}}}  
\newcommand{\Lp}{{L}}            
\newcommand{\cM}{{\mathcal{M}}} 
\newcommand{\fM}{{\mathfrak{M}}} 
\newcommand{\cT}{{\mathcal{T}}}
\newcommand{\cP}{{\mathcal{P}}}  

\newcommand{\Lyap}{{\mathcal{V}}}  

\newcommand{\Zplus}{\mathds{Z}_+}
\newcommand{\RR}{\mathds{R}}
\newcommand{\NN}{\mathds{N}}

\newcommand{\Rd}{{\mathds{R}^{d}}}
\DeclareMathOperator{\Exp}{\mathbb{E}}

\newcommand{\D}{\mathrm{d}}

\newcommand{\cD}{\mathcal{D}} 
\newcommand{\Sob}{{\mathscr W}}    
\newcommand{\Sobl}{{\mathscr W}_{\text{loc}}} 

\newcommand{\df}{:=}
\newcommand{\transp}{^{\mathsf{T}}}

\DeclareMathOperator*{\trace}{Tr}

\newcommand{\order}{{\mathscr{O}}}
\newcommand{\sorder}{{\mathfrak{o}}}
\newcommand{\grad}{\nabla}
\newcommand{\uuptau}{{\Breve\uptau}}

\newcommand{\abs}[1]{\lvert#1\rvert}
\newcommand{\norm}[1]{\lVert#1\rVert}

\usepackage{color}
\definecolor{dmagenta}{rgb}{.4,.1,.5}
\definecolor{dblue}{rgb}{.0,.0,.5}
\definecolor{mblue}{rgb}{.0,.0,.7}
\definecolor{ddblue}{rgb}{.0,.0,.4}
\definecolor{dred}{rgb}{.7,.0,.0}
\definecolor{dgreen}{rgb}{.0,.5,.0}
\definecolor{Eeom}{rgb}{.0,.0,.5}

\begin{document}
\title[Robustness of Markov Chain Approximations for Controlled Diffusions]
{Discrete-Time Approximations of Controlled Diffusions with Infinite Horizon Discounted and Average Cost}

\author[Somnath Pradhan]{Somnath Pradhan$^\dag$}
\address{$^\dag$Department of Mathematics, Indian Institute of Science Education and Research Bhopal, MP-462066, India}
\email{somnath@iiserb.ac.in}

\author[Serdar Y\"{u}ksel]{Serdar Y\"{u}ksel$^{\ddag}$}
\address{$^\ddag$Department of Mathematics and Statistics,
Queen's University, Kingston, ON, Canada}
\email{yuksel@queensu.ca}

\begin{abstract}
We present discrete-time approximation of optimal control policies for infinite horizon discounted/ergodic control problems for controlled diffusions in $\Rd$\,. In particular, our objective is to show near optimality of optimal policies designed from the approximating discrete-time controlled Markov chain model, for the discounted/ergodic optimal control problems, in the true controlled diffusion model (as the sampling period approaches zero). To this end, we first construct suitable discrete-time controlled Markov chain models for which one can compute optimal policies and optimal values via several methods (such as value iteration, convex analytic method, reinforcement learning etc.). Then using a weak convergence technique, we show that the optimal policy designed for the discrete-time Markov chain model is near-optimal for the controlled diffusion model as the discrete-time model approaches the continuous-time model. This provides a practical approach for finding near-optimal control policies for controlled diffusions. Our conditions complement existing results in the literature, which have been arrived at via either probabilistic or PDE based methods.
\end{abstract}


\maketitle


\section{Introduction}
In this paper, we study discrete-time approximations of optimal control policies for controlled diffusion models under infinite horizon discounted/ergodic cost criteria\,. 

\subsection{Literature Review}
In the study of controlled difussions, researchers have extensively studied the problems of existence, uniqueness and verification of optimality of stationary Markov policies see e.g., \cite{BS86}, \cite{BB96} (discounted cost)  \cite{AA12}, \cite{AA13}, \cite{BG88I}, \cite{BG90b} (ergodic cost) and references therein. For a book-length exposition of this topic see e.g., \cite{ABG-book}. 

The analytical properties of solution to the Hamilton-Jacobi-Bellman (HJB) equation are important for understanding the behavior of optimal-value/ optimal control policies. However, for many important classes of problems, little is known on the analytical properties of solutions of the associated HJB equations. In view of this, in order to compute the value-function/ optimal control policies of such problems, many numerical schemes have been constructed, see e.g., \cite{BR-02}, \cite{BJ-06}, \cite{JPR-19P}, \cite{kushner1990numerical,kushner2001numerical, kushner2012weak}.

For discounted cost and ergodic cost (for reflected diffusion processes in a smooth bounded domain), there is an extensive literature, and we refer the reader to \cite{kushner1990numerical,kushner2001numerical, kushner2012weak,fleming2006controlled} for a detailed analysis and review. In our paper, in addition to a review, we also relax some of the technical regularity conditions studied in these papers. In particular for the ergodic cost case, we establish discrete-time approximation results for a general class of controlled diffusions in $\Rd$\,.

In the continuous-time literature, most of the approximation results focus on finite horizon or discounted cost criteria see, e.g., \cite{BR-02}, \cite{BJ-06}, \cite{KD92}, \cite{KH77}, \cite{KH01}, \cite{KN98A}, \cite{KN2000A}\,, though the ergodic control and control up to an exit time criteria have also been studied \cite{KD92,kushner2014partial}. For finite horizon criteria, a commonly adopted approach of approximating controlled diffusions by a sequence of discrete time Markov chains via weak convergence methods was studied by Kushner and Kushner and Dupuis, see \cite{KD92}, \cite{KH77}, \cite{KH01}\,. These works deal with numerical procedures to construct near optimal  control policies for controlled diffusion models by approximating the space of (open-loop adapted) relaxed control policies with those that are piece-wise constant, and by considering the weak convergence of approximating probability measures on the path space to the measure on the continuous-time limit. It is shown in \cite{KD92}, \cite{KH77}, \cite{KH01} that if the constructed controlled Markov chain satisfies a certain ``consistency" condition at the discrete-time sampling instants, then the state process and the corresponding value function asymptotically approximates the continuous time state process and the associated value function. This approach has been referred to as the {\it weak convergence} approach.

Making use of the above weak convergence approach and Girsanov's change of measure technique, when the time horizon is finite, the authors in \cite{pradhanyuksel2023DTApprx} established the existence of an optimal policy and discrete-time approximation for controlled diffusion under a variety of information structures through a unified perspective. These include fully observed models, partially observed models and multi-agent models with decentralized information structures.

In an alternative program, building on finite difference approximations for HJB equations utilizing their regularity properties, Krylov \cite{KN98A}, \cite{KN2000A} established the convergence rate of for such approximation techniques, where finite difference approximations are studied to arrive at stability results. In particular, some estimates for the error bound of finite-difference approximation schemes in the problem of finding viscosity or probabilistic solutions to degenerate Bellman equations are established. The proof technique is based on mean value theorems for stochastic integrals (as in \cite{KN2001AA}), obtained on the basis of elementary properties of the associated Bellman's equations\,. Also, for controlled non-degenerate diffusion processes, it is shown in \cite{KN99AA} that using policies which are constant on intervals of length $h^2$, one can approximate the value function with errors of order $h^{\frac{1}{3}}$\,. In \cite{BR-02}, \cite{BJ-06} Barles et. al. improved the error bounds obtained in \cite{KN98A}, \cite{KN2000A}, \cite{KN99AA}; along this note a recent further analysis is presented in \cite{jakobsen2019improved} where both stochastic analysis and PDE methods are utilized.

Such numerical methods have long been a standard approach for tackling the HJB equations. These methods primarily focus on discretizing the non-linear HJB equations using techniques such as finite difference and finite element methods. The basic idea is to partition the state space into a grid and approximate the derivatives in the PDEs using finite differences or to represent the solution in terms of basis functions for finite elements.

However, as noted in \cite{KBHIFAC23}, these numerical methods face inherent challenges, particularly as the dimensionality of the state space increases. As the number of dimensions grows, the computational burden becomes increasingly unmanageable. For instance, extending methods beyond three dimensions often requires sophisticated numerical techniques.

In contrast, probabilistic methods have emerged as an alternative to the numerical PDE methods. By leveraging the Markov chain approximation, these methods offer a more intuitive and flexible framework for exploring the value function in high-dimensional spaces. Additionally, probabilistic approaches offer a variety of techniques to approximate the optimal value and optimal policy of the discrete-time model including Reinforcement Learning (RL) and Empirical Dynamic Programming (EDP). Reinforcement Learning is a model-free approach that allows agents to learn optimal strategies through interaction with the environment; RL algorithms such as Q-learning and policy gradient methods have gained popularity for their ability to approximate the value function directly from observed data, making them particularly suitable for complex, high-dimensional problems. Empirical Dynamic Programming (EDP) is an extension of traditional dynamic programming that estimates the value function using sampled trajectories from the system.

\subsection{Contributions}
 
In this article, we propose a probabilistic framework for discrete-time approximations of optimal control policies in controlled diffusion models, specifically addressing both infinite horizon discounted and ergodic cost criteria. Our approach deviates from the traditional PDE-based methods prevalent in the literature, offering complementary insights and conditions that enhance the understanding of optimal control in these settings. Unlike methods based on PDEs, our approach incorporates a probabilistic structure that can yield solutions under different sets of assumptions and regularity conditions. This distinction is crucial as it allows for the treatment of problems that may not fit neatly into the PDE framework. In particular, in order to construct a suitable approximating discrete-time Markov chain via finite deference approximation approach, the authors in \cite[Section~5.3.1, p. 108]{KD92} assumed that for all $i=1,\dots , d$, $$a_{i,i} - \sum_{j, j\neq i}\abs{a_{i,j}} \geq 0,$$ where $a(x)= \frac{1}{2}\upsigma(x)\upsigma(x)^{\transp}$\,. This condition may not be satisfied by the diffusion matrix $\upsigma$ in general\,. Furthermore, our approach allows for less stringent regularity conditions, for example our cost function does not require a global Lipschitz condition \cite{jakobsen2019improved}. 

The approach for finite horizon problem is different from the infinite horizon criteria we study in this paper. A critical argument in our paper utilizes near optimality of Lipschitz policies for such problems presented in our recent work \cite{PYNearsmoothSCL24}, building on \cite{pradhan2022near}. 

We emphasize that the average cost setup has not been studied in the generality we present and our results on the discounted cost criteria complements the prior studies. 

Our analysis allows one to take advantage of the relatively more mature theory of discrete-time stochastic control, as there is an established theory on existence, approximation, and learning theory for optimal discrete-time stochastic control\,. Notably, recently Bayraktar and Kara \cite{bayraktar2022approximate} obtained near optimality results for time and space discretized Q-learning for controlled diffusions (building on quantized Q-learning for standard Borel MDPs \cite{KSYContQLearning}). 


As our primary contribution, we show that for a continuous-time model 
 \begin{align*}
 dX_t = b(X_t,U_t)dt + \upsigma(X_t)dW_t\,,
 \end{align*} an optimal policy designed for the time discretized model
  \begin{align*}
 X_{(k+1)h} =& X_{kh} + \int_{kh}^{(k+1)h} b(X_{s},U_{kh})ds  + \int_{kh}^{(k+1)h} \upsigma(X_{s}) dW_s
 \end{align*} is near optimal for discounted cost criteria (see Theorems \ref{TD1.2} and \ref{TD1.3}) as well as for average cost criteria (see Theorems \ref{TContValue1A} and \ref{TNearOpt1A}).
\subsection*{Notation:}
\begin{itemize}
\item For any set $A\subset\RR^{d}$, by $\uptau(A)$ we denote \emph{first exit time} of the process $\{X_{t}\}$ from the set $A\subset\RR^{d}$, defined by
\begin{equation*}
\uptau(A) \,\df\, \inf\,\{t>0\,\colon X_{t}\not\in A\}\,.
\end{equation*}
\item $\sB_{r}$ denotes the open ball of radius $r$ in $\RR^{d}$, centered at the origin,
\item $\uptau_{r}$, $\uuptau_{r}$ denote the first exist time from $\sB_{r}$, $\sB_{r}^c$ respectively, i.e., $\uptau_{r}\df \uptau(\sB_{r})$, and $\uuptau_{r}\df \uptau(\sB^{c}_{r})$.
\item By $\trace S$ we denote the trace of a square matrix $S$.
\item For any domain $\cD\subset\RR^{d}$, the space $\cC^{k}(\cD)$ ($\cC^{\infty}(\cD)$), $k\ge 0$, denotes the class of all real-valued functions on $\cD$ whose partial derivatives up to and including order $k$ (of any order) exist and are continuous.
\item $\cC_{\mathrm{c}}^k(\cD)$ denotes the subset of $\cC^{k}(\cD)$, $0\le k\le \infty$, consisting of functions that have compact support. This denotes the space of test functions.
\item $\cC_{b}(\Rd)$ denotes the class of bounded continuous functions on $\Rd$\,.
\item $\cC^{k}_{0}(\cD)$, denotes the subspace of $\cC^{k}(\cD)$, $0\le k < \infty$, consisting of functions that vanish in $\cD^c$.
\item $\Lp^{p}(\cD)$, $p\in[1,\infty)$, denotes the Banach space
of (equivalence classes of) measurable functions $f$ satisfying
$\int_{\cD} \abs{f(x)}^{p}\,\D{x}<\infty$.
\item $\Sob^{k,p}(\cD)$, $k\ge0$, $p\ge1$ denotes the standard Sobolev space of functions on $\cD$ whose generalized derivatives up to order $k$ are in $\Lp^{p}(\cD)$, equipped with its natural norm (see, \cite{Adams})\,.
\item  If $\mathcal{X}(Q)$ is a space of real-valued functions on $Q$, $\mathcal{X}_{\mathrm{loc}}(Q)$ consists of all functions $f$ such that $f\varphi\in\mathcal{X}(Q)$ for every $\varphi\in\cC_{\mathrm{c}}^{\infty}(Q)$. In a similar fashion, we define $\Sobl^{k, p}(\cD)$.
\end{itemize}

\section{Description of the problem}\label{PD} Let $\Act$ be a convex compact subset of a Euclidean space $\RR^N$ and $\pV=\cP(\Act)$ be the space of probability measures on  $\Act$ with topology of weak convergence. Let $$b : \Rd \times \Act \to  \Rd, $$ $$ \sigma : \Rd \to \RR^{d \times d},\, \sigma = [\sigma_{ij}(\cdot)]_{1\leq i,j\leq d},$$ be given functions. We consider a stochastic optimal control problem whose state is evolving according to a controlled diffusion process given by the solution of the following stochastic differential equation (SDE)
\begin{equation}\label{E1.1}
\D X_t \,=\, b(X_t,U_t) \D t + \upsigma(X_t) \D W_t\,,
\quad X_0=x\in\Rd.
\end{equation}
Where 
\begin{itemize}
\item
$W$ is a $d$-dimensional standard Wiener process, defined on a complete probability space $(\Omega, \sF, \mathbb{P})$.
\item 
 We extend the drift term $b : \Rd \times \pV \to  \Rd$ as follows:
\begin{equation*}
b (x,\mathrm{v}) = \int_{\Act}b (x,\zeta)\mathrm{v}(\D \zeta), 
\end{equation*}
for $\mathrm{v}\in\pV$.
\item
$U$ is a $\pV$ valued process satisfying the following non-anticipativity condition: for $s<t\,,$ $W_t - W_s$ is independent of
\begin{align*}
\sF_s := &\,\,\mbox{the completion of}\,\,\, \sigma(X_0, U_r, W_r : r\leq s)\nonumber\\
&\,\,\,\mbox{relative to} \,\, (\sF, \mathbb{P})\,.
\end{align*}
\end{itemize}
The process $U$ is called an \emph{admissible} control, and the set of all admissible controls is denoted by $\Uadm$ (see, \cite{BG90}).

To ensure existence and uniqueness of solutions of \cref{E1.1}, we impose the following assumptions on the drift $b$ and the diffusion matrix $\upsigma$\,. 
\begin{itemize}
\item[\hypertarget{A1}{{(A1)}}]
\emph{Lipschitz continuity:\/}
The function
$\upsigma\,=\,\bigl[\upsigma^{ij}\bigr]\colon\RR^{d}\to\RR^{d\times d}$,
$b\colon\Rd\times\Act\to\Rd$ are uniformly bounded and Lipschitz continuous in $x$, $(x, \zeta)$ respectively. In other words, for some constant $C_{0}>0$, we have
\begin{align*}
\abs{b(x,\zeta_1) - b(y, \zeta_2)}^2 &+ \norm{\upsigma(x) - \upsigma(y)}^2 \nonumber\\
&\,\le\, C_{0}\,\left(\abs{x-y}^2 + \abs{\zeta_1- \zeta_2}^2\right)\,,
\end{align*}
for all $x,y\in \Rd$ and $\zeta_1, \zeta_2\in\Act$, where $\norm{\upsigma}\df\sqrt{\trace(\upsigma\upsigma\transp)}$\,.
\medskip
\item[\hypertarget{A2}{{(A2)}}]
\emph{Nondegeneracy:\/}
For each $R>0$, there exists positive constant $C_{R}$ (depending on $R$) such that
\begin{equation*}
\sum_{i,j=1}^{d} a^{ij}(x)z_{i}z_{j}
\,\ge\,C^{-1}_{R} \abs{z}^{2} \qquad\forall\, x\in \sB_{R}\,,
\end{equation*}
and for all $z=(z_{1},\dotsc,z_{d})\transp\in\RR^{d}$,
where $a\df \frac{1}{2}\upsigma \upsigma\transp$.
\end{itemize}

By a Markov control we mean an admissible control of the form $U_t = v(t,X_t)$ for some Borel measurable function $v:\RR_+\times\Rd\to\pV$. The space of all Markov controls is denoted by $\Um$\,.
If the function $v$ is independent of $t$, then $v$ is called a stationary Markov control. The set of all stationary Markov controls is denoted by $\Usm$. A policy $v\in \Usm$ is said to be a deterministic stationary Markov policy if $v(x) = \delta_{f(x)}$ for some measurable map $f:\Rd\to \Act$. Let $\Udsm$ be space of all deterministic stationary Markov policies\,. From \cite[Section~2.4]{ABG-book}, we have that the set $\Usm$ is metrizable under the following Borkar topology, which defines a compact metric under which a sequence $v_n\to v$ in $\Usm$ if and only if
\begin{align*}
&\lim_{n\to\infty}\int_{\Rd}f(x)\int_{\Act}g(x,\zeta)v_{n}(x)(\D \zeta)\D x \nonumber\\
& = \int_{\Rd}f(x)\int_{\Act}g(x,\zeta)v(x)(\D \zeta)\D x
\end{align*}
for all $f\in L^1(\Rd)\cap L^2(\Rd)$ and $g\in \cC_b(\Rd\times \Act)$\,. It is well known that under the hypotheses \hyperlink{A1}{{(A1)}}--\hyperlink{A2}{{(A2)}}, for any admissible control \cref{E1.1} has a unique strong solution \cite[Theorem~2.2.4]{ABG-book}, and under any stationary Markov strategy \cref{E1.1} has a unique strong solution which is a strong Feller (therefore strong Markov) process \cite[Theorem~2.2.12]{ABG-book}.

Let $c\colon\Rd\times\Act \to \RR_+$ be the \emph{running cost} function. We assume that $c$ is bounded, jointly continuous in $(x, \zeta)$ and locally Lipschitz continuous in its first argument uniformly with respect to $\zeta\in\Act$. We extend $c\colon\Rd\times\pV \to\RR_+$ as follows: for $\pv \in \pV$
\begin{equation*}
c(x,\pv) := \int_{\Act}c(x,\zeta)\pv(\D\zeta)\,.
\end{equation*}
\textbf{Discounted Cost Criterion:} For $U \in\Uadm$, the associated \emph{$\alpha$-discounted cost} is given by
\begin{equation}\label{EDiscost}
\cJ_{\alpha}^{U}(x, c) \,\df\, \Exp_x^{U} \left[\int_0^{\infty} e^{-\alpha s} c(X_s, U_s) \D s\right],\quad x\in\Rd\,,
\end{equation} where $\alpha > 0$ is the discount factor
and $X_{\cdot}$ is the solution of the \cref{E1.1} corresponding to $U\in\Uadm$ and $\Exp_x^{U}$ is the expectation with respect to the law of the process $X_{\cdot}$ with initial condition $x$. Here the controller tries to minimize \cref{EDiscost} over the set of admissible controls $U\in \Uadm$\,.
A control $U^{*}\in \Uadm$ is said to be an optimal control if for all $x\in \Rd$ 
\begin{equation}\label{OPDcost}
\cJ_{\alpha}^{U^*}(x, c) = \inf_{U\in \Uadm}\cJ_{\alpha}^{U}(x, c) \,\,\, (\df \,\, V_{\alpha}(x))\,.
\end{equation}\\ \\
\textbf{Ergodic Cost Criterion:} For $U\in\Uadm$, the associated ergodic cost is defined as
\begin{equation}\label{ECost1}
\sE_{x}(c, U) = \limsup_{T\to \infty}\frac{1}{T}\Exp_x^{U}\left[\int_0^{T} c(X_s, U_s) \D{s}\right]\,.
\end{equation} and the optimal value is defined as
\begin{equation}\label{ECost1Opt}
\sE^*(c) \,\df\, \inf_{x\in\Rd}\inf_{U\in \Uadm}\sE_{x}(c, U)\,.
\end{equation}
Then a control $U^*\in \Uadm$ is said to be optimal if we have 
\begin{equation}\label{ECost1Opt1}
\sE_{x}(c, U^*) = \sE^*(c)\quad \text{for all} \,\,\, x\in \Rd\,.
\end{equation}
Associated to the controlled diffusion model \cref{E1.1}, we define a family of operators $\sL_{\zeta}$ mapping $\cC^2(\Rd)$ to $\cC(\Rd)$ by
\begin{equation}\label{E-cI}
\sL_{\zeta} f(x) \,\df\, \trace\bigl(a(x)\grad^2 f(x)\bigr) + \,b(x,\zeta)\cdot \grad f(x)\,, 
\end{equation}
for $\zeta\in\Act$, \,\, $f\in \cC^2(\Rd)\cap\cC_b(\Rd)$\,.
For $\pv \in\pV$ we extend $\sL_{\zeta}$ as follows:
\begin{equation}\label{EExI}
\sL_\pv f(x) \,\df\, \int_{\Act} \sL_{\zeta} f(x)\pv(\D \zeta)\,.
\end{equation} For $v \in\Usm$, we define
\begin{equation}\label{Efixstra}
\sL_{v} f(x) \,\df\, \trace(a\grad^2 f(x)) + b(x,v(x))\cdot\grad f(x)\,.
\end{equation}
\subsection*{Approximating Controlled Markov Chains:}
Let $\Rd$ be the state space, and $\Act$ be the control action space of the controlled Markov chain. 

We define the transition probabilities as follows: For any $k\in\Zplus$,  distribution of the state $X_k^h$ conditioned on the past state and action variables, is determined by the diffusion process (\ref{E1.1}), such that, conditioned on $(X_{k-1}^h,\dots,X_0^h,U^h_{k-1},\dots,U_0^h)$, $X_k^h$ has the same distribution as 
\begin{align}\label{sampled_MDP}
X_{kh} &= X_{k-1}^h +\int_{(k-1)h}^{kh}b(X_s,U^h(s))ds +\int_{(k-1)h}^{kh}\upsigma(X_s)dW_s
\end{align} 
where $U^h(s)=U_{k-1}^h$ for all $(k-1)h \leq s < kh$ (that is the control $U^h$ is piece-wise constant in time). Hence, for any $A\in\cB(\Rd)$
\begin{align*}
Pr(X_k^h\in A|X^h_{[0,k-1]},U^h_{[0,k-1]})=\mathcal{T}_h(A|X^h_{k-1},U^h_{k-1})
\end{align*}
where ${X^h}_{[0,k-1]},{U^h}_{[0,k-1]}:=X_0^h,\dots,X^h_{k-1},U^h_0,\dots,U^h_{k-1}$, such that 
\begin{align*}
X_k^h \sim \mathcal{T}_h(dx_k|X^h_{k-1},U^h_{k-1})
\end{align*}
where $X_k^h$ determined by (\ref{sampled_MDP}) and where  $\mathcal{T}_h$ is the transition kernel of the Markov chain which is a stochastic kernel from $\Rd\times
\Act$ to $\Rd$.

For the discrete-time model, an {\em admissible policy} is a sequence of control functions $\{U^h_k,\, k\in \Zplus\}$ such that $U^h_k$ is measurable with respect to the $\sigma$-algebra generated by the information variables
$
I_k^h=\{X_{[0,k]}^h,U_{[0,k-1]}^h\},\,\, k \in \NN,\,\, I_0^h=\{X_0^h\},
$ that is
\begin{equation}
\label{eq_control}
U_k^h=v_k^h(I_k^h),\quad k\in \Zplus,\nonumber
\end{equation}
where $v_k^h$ is a $\pV$-valued measurable function for $k\in \Zplus$\,. We define $\Uadm^h$ to be the set of all such admissible policies. Let $\Usm^h := \{U^h\in \Uadm^h: U_k^h = v^h(X_k^h)\,\,\text{for some measurable map}\,\, v^h:\Rd\to\pV\}$ be the set of all stationary Markov policies\,. A policy $v^h\in \Usm^h$ is said to be a deterministic stationary Markov policy if $v^h(x) = \delta_{f(x)}$ for some measurable map $f:\Rd\to \Act$. Let $\Udsm^h$ be the space of all deterministic stationary Markov policies\,.

We also define a stage-wise cost function $c_h:\Rd\times\Act\to \RR_+$ such that for any $(x,\zeta)\in\Rd\times\Act$
\begin{align*}
c_h(x,\zeta):=c(x,\zeta)\times h
\end{align*} where $c$ is the cost function of the diffusion model\,. We are interested in the following cost evaluation criteria\,.\\
\textbf{Discrete-time Discounted Cost:} For each $U^h\in \Uadm^h$, the associated discounted cost of the approximating discrete-time model is given by
\begin{equation}\label{E2.1APM}
\cJ_{\alpha,h}^{U^h}(x, c) \,\df\, \Exp_x^{U^h} \left[\sum_{k=0}^{\infty} \beta^k c_h(X_k^h, U_k^h) | X_{0}^h = x \right]\,,
\end{equation} for $x\in\Rd$, where $\beta \df e^{-\alpha h}$\,. The optimal cost is defined as
\begin{equation}\label{E2.1APM1}
\cJ_{\alpha,h}^{*}(x, c) \,\df\, \inf_{U^h\in \Uadm^h}\cJ_{\alpha,h}^{U^h}(x, c)\,.
\end{equation}
\textbf{Discrete-time Ergodic Cost:} For each $U^h\in \Uadm^h$, the associated infinite horizon average cost function is defined as
\begin{align}\label{MDP1_cost}
\sE_h(x,U^h):=\limsup_{N\to\infty}\frac{1}{Nh}\Exp_{x}^{U^h}\left[\sum_{k=0}^{N-1} c_h(X_k^h,U_k^h)\right]
\end{align}
The optimal ergodic cost function is defined as
\begin{align}\label{MDP1_optcost}
\sE_h^*(x):= \inf_{U^h\in \Uadm^h} \sE_h(x,U^h)\,.
\end{align}
\begin{remark}\label{R1}
An important property of the MDP model we will make use of is the following one: suppose that we are given an admissible policy $U^h\in\Uadm^h$ defined for the MDP, we define the following continuous-time interpolated control process $U^{h}(\cdot)$ defined as
\begin{align}\label{pccontrol}
U^{h}(t)= U^h_k \text{ for } t\in[kh, (k+1)h)
\end{align}
which is a piece-wise constant control process. Then, the controlled Markov chain state process $X_k^h$ under the policy $U^h$ and the controlled diffusion process  $\tilde{X}^h_t$ under the control process defined in (\ref{pccontrol}) have the same distributions at the sampling instances if they start from the same initial points, that is, for any $k\in \Zplus$, $$X_k^h\sim \tilde{X}^h_{kh}\,.$$
\end{remark}
For any $U^h\in \Uadm^h$ the interpolated continuous time policy $U^h(\cdot)$ is defined as in (\ref{pccontrol}). The set of all interpolated continuous time admissible policies is denoted by $\bar{\Uadm}^h$ and set of all interpolated continuous time stationary Markov strategies is denoted by $\bar{\Uadm}_{\mathsf{sm}}^h$\,. Also, for any discrete-time Markov chain $\{X^h_k\}$, the associated continuous-time interpolated process $X^{h}(\cdot)$ is given by
\begin{align}\label{pcState}
X^{h}(t)= X^h_k \text{ for } t\in[kh, (k+1)h)\,.
\end{align} Thus, the sample paths of the continuous-time interpolated process $X^{h}(\cdot)$ are right continuous with left limits. This leads us to consider the function space 
\begin{align*}
\sD(\Rd; 0, \infty)& \df \{\omega:[0, \infty)\to \Rd \mid \omega(\cdot) \,\,\text{is right} \nonumber\\
&\text{continuous and have left limits at every}\,\, t > 0\}\,.
\end{align*} We will consider the space $\sD(\Rd; 0, \infty)$ endowed with the Skorokhod topology (for details see \cite[Section~16, p. 166]{PBill-book})\,.

Let $\cC_{BL}(\Rd)$ be the space of bounded Lipschitz continuous functions on $\Rd$, i.e., 
\begin{align}\label{BoundLips}
\cC_{BL}(\Rd) & \df \{f\in \cC(\Rd): f \,\,\text{is Lipschitz continuous and}\nonumber\\
&\,\,\norm{f}_{BL}\df\sup_{x\neq y}\frac{|f(x) - f(y)|}{|x-y|}  + \sup_{x\in \Rd} |f(x)| < \infty\}\,. 
\end{align}

In order to utilize the weak convergence technique, we introduce the relaxed control representation of the control policies (for more details see \cite[Section~9.5]{KD92})\,. 
Let $\fB(\Act\times [0, \infty))$ be the $\sigma$-algebra of Borel subsets of $\Act\times [0, \infty)$\,. Then for any Borel measure $\hat{m}$ on $\fB(\Act\times [0, \infty))$, satisfying $\hat{m}(\Act\times [0, t]) = t$ for all $t\geq 0$, it is easy to see that there exists a measure $\hat{m}_{t}$ on $\fB(\Act)$ such that $\hat{m}(\D\zeta, \D t) = \hat{m}_t(\D\zeta) \D t$\,. Let 
\begin{align*}
&\fM(\infty) \df \{\hat{m}: \hat{m}\,\,\text{is a Borel measure on}\,\,\fB(\Act\times [0, \infty))\nonumber\\
&\,\,\text{satisfying}\,\, \hat{m}(\Act\times [0,t]) = t, \text{for any}\,\, t\geq  0,\,\, \text{with}\,\, \hat{m}_t\in \cP(\Act)\}\,
\end{align*} The space $\fM(\infty)$ can be metrized by using the Prokhorov metric over $\cP(\Act\times [0, n])$ for $n\in\NN$\,. A sequence $\hat{m}^{n}$ converges to $\hat{m}$ in $\fM(\infty)$, if the normalized restriction of $\hat{m}^{n}$ converges weakly to the normalized restriction $\hat{m}$ on $\cP(\Act\times [0, n])$ for each $n\in \NN$\,. In particular, a sequence $\hat{m}^{n}$ converge to $\hat{m}$ in $\fM(\infty)$ if for any $\phi\in \cC_c(\Act\times [0, \infty))$ we have (this is known as compact weak topology)
$$\int_{\Act\times [0, \infty)} \phi(\zeta, t)\hat{m}^{n}(\D \zeta, \D t)\to \int_{\Act\times [0, \infty)} \phi(\zeta, t)\hat{m}(\D \zeta, \D t)\,.$$

Since $\cP(\Act\times [0, n])$ is complete, separable and compact for each $n\in\NN$, the space $\fM(\infty)$ inherits those properties\,. 

Now, define
\begin{align*}
\cM(\infty) =& \bigg\{m: m\,\,\text{is a random measure taking values in}\nonumber\\
&\,\, \fM(\infty)\,\, \text{satisfying}\,\,\text{for any}\,\,0\leq s <t < \infty,\nonumber\\
&\,\,  m_{[0,s]}\,\,\text{is independent of}\,\, W_t - W_s\,\bigg\}\,.
\end{align*} Since the range space $\fM(\infty)$ is compact \footnote{ By a diagonal argument, every sequence would have a converging subsequence because every compact restriction has a convergent subsequence; thus, the space is compact because it is a metric space.}, we have that any sequence in $\cM(\infty)$ is tight\,.   

In this article our primary goal is to show the following robustness result with respect to discrete-time approximation of the controlled diffusion models: 
\begin{itemize}
\item[•] {\it For discounted cost criterion:} We establish continuity of value functions and provide sufficient conditions which ensure near-optimality of optimal controls designed from the discretized models in controlled diffusion model (see Theorem~\ref{TD1.2} (continuity), Theorem~\ref{TD1.3} (near optimality)).
\item[•] {\it For ergodic cost criterion:} Under Lyapunov type stability assumptions we establish the continuity of value functions and exploiting the continuity results we derive the near-optimality of ergodic optimal controls designed for discrete-time models applied to actual continuous-time systems (see Theorem~\ref{TContValue1A} (continuity), Theorem~\ref{TNearOpt1A} (near optimality)).
\end{itemize} 

\section{Supporting Results}
In this section, we show that the continuous-time interpolated process of the sampled discrete time Markov chain converges weakly to the diffusion process as the parameter of discretization approaches to zero\,. 
\begin{theorem}\label{TConvChain} 
Suppose Assumptions \hyperlink{A1}{{(A1)}}--\hyperlink{A2}{{(A2)}} hold. Let $v^h\in \Uadm_{\mathsf{sm}}^h$ and $\{X_n^h\}_{n\geq 1}$ be the associated chain with initial condition $x_0^h$. If $x_0^h\to x$. Then the continuous time interpolated process $(X^h(\cdot), v^h(\cdot))$ is tight. Let $(X_{\cdot}, U_{\cdot})$ be a limit of a weakly convergent sub-sequence then $U\in \Uadm$ and $X$ is a solution of \cref{E1.1} under $U\in \Uadm$\,. 
\end{theorem}

\begin{proof}
Since the range space $\fM(\infty)$ is compact, viewing $v^{h}(\cdot)$ as an element of $\cM(\infty)$, it is easy to see that the sequence $v^{h}(\cdot)$ is tight\,. Hence, along a sub-sequence $v^{h}(\cdot)$ converges weakly to some $U\in\Uadm$ (since weak convergence preserves non-anticipativity, see \cite[Lemma~1.2, Lemma~1.3]{Bor-book}). 

\textbf{Step-1:} Let $\tilde{X}^h_t$ be the solution of (\ref{E1.1}) under the control policy $v^h(\cdot)$ with initial condition $x^h_0$\,. Then for any finite stopping time $\tau$ (with respect to the filtration generated by $X^{h}(\cdot)$) and any positive constant $\delta$, we have 
\begin{align}\label{EtightA}
&\Exp\big[|X^{h}(\tau + \delta) - X^{h}(\tau)|^2\big] \nonumber\\
\leq & \Exp\big[|\int_{\tau}^{\tau +\delta} \left(b(\tilde{X}^h_s, v^h(\tilde{X}^h_s)) \D s + \upsigma(\tilde{X}^h_s)dW_s\right)|^2\big]\nonumber\\
\leq & 2(\|b\|_{\infty}^2 \delta +  \|\upsigma\|_{\infty}^2)\delta\,.
\end{align} Therefore, by Aldous's tightness criterion (see \cite[Theorem~16.10]{PBill-book}) the continuous-time interpolated process $X^{h}(\cdot)$ is tight in $\sD(\Rd; 0, \infty)$\,. 
Since $(X^h(\cdot),v^h(\cdot))$ is tight, let $(\hat{X}(\cdot),U)$ be the limit of a weakly convergent sub-sequence. Thus, by Skorohod's representation theorem (see \cite[Theorem~6.7]{PBill-book}), over a common probability space we have $X^h(\cdot, \omega)\to \hat{X}_{\cdot}(\omega)$ (almost surely in $\omega$ under Skorokhod metric) and $v^n(\cdot , \omega) \to U_{\cdot}(\omega)$ weakly as a $\cP([0,t]\times\Act)$-valued probability measure sequence, almost surely (in $\omega$), that is
\begin{align}\label{ETConvChain1AB}
&\lim_{h\to\infty}\int_{0}^{t}\int_{\Act}\hat{f}(s, \zeta)v^{h}(s,\omega)(\D \zeta)\D s
= \int_{0}^{t}\int_{\Act}\hat{f}(s, \zeta)U_s(\omega)(\D \zeta)\D s
\end{align} almost surely in $\omega$, for any $\hat{f}\in \cC_{b}([0,t]\times \Act)$ and $t \geq 0$\,.

\textbf{Step-2:} Let $f\in\cC_c^{2}(\Rd)$ and $\phi \in \cC_c^{\infty}(\Rd)$ with $\phi = 0$ for $|x|>1$ and $\int_{\Rd}\phi(x) d x = 1$\,. For $\epsilon>0$, define $\phi^{\epsilon}(x) = \epsilon^{-d}\phi(\frac{x}{\epsilon})$ and let $f^{\epsilon}(x) := \phi^{\epsilon}* f(x)$ (convolution with respect to $x$). Since, $f\in\cC_c^{2}(\Rd)$ it is well known that $\frac{\partial f^{\epsilon}(x)}{\partial x_{i}} =  \phi^{\epsilon}*\frac{\partial f(x)}{\partial x_{i}}(x)$, $\frac{\partial^2 f^{\epsilon}(x)}{\partial x_{i}\partial x_{j}} =  \phi^{\epsilon}*\frac{\partial^2 f}{\partial x_{i}\partial x_{j}}(x)$ and $\|f^{\epsilon}\|_{\infty} \leq \|f\|_{\infty}$, $\|\frac{\partial f^{\epsilon}}{\partial x_{i}}\|_{\infty} \leq \|\frac{\partial f}{\partial x_{i}}\|_{\infty}$, $\|\frac{\partial^2 f^{\epsilon}}{\partial x_{i}\partial x_{j}}\|_{\infty} \leq \|\frac{\partial^2 f}{\partial x_{i}\partial x_{j}}\|_{\infty}$ for all $i,j= 1,\dots , d$\,. Moreover, as $\epsilon\to 0$, we have   
\begin{equation}\label{ETConvChain1Molli}
\|f^{\epsilon} - f\|_{\infty} + \|\frac{\partial f^{\epsilon}}{\partial x_{i}} - \frac{\partial f}{\partial x_{i}}\|_{\infty} + \|\frac{\partial^2 f^{\epsilon}}{\partial x_{i}\partial x_{j}} - \frac{\partial^2 f}{\partial x_{i}\partial x_{j}}\|_{\infty} \to 0\,.   
\end{equation}

Now, in view of Remark~\ref{R1}, by It$\hat{\rm o}$-Krylov formula, we obtain the following
\begin{align}\label{ETConvChain1AA} 
&\Exp_{k}^{h}\left[f^{\epsilon}(X_{k+1}^h)\right] - f^{\epsilon}(X_{k}^h) \nonumber\\
& = \Exp_{k}^{h}\bigg[\int_{kh}^{(k+1)h}\bigg( \trace\bigl(a(\tilde{X}^h_s)\grad^2 f^{\epsilon}(\tilde{X}^h_s)\bigr) + \int_{\Act}\,b(\tilde{X}^h_s,\zeta)\cdot \grad f^{\epsilon}(\tilde{X}^h_s)v^h(s)(\D \zeta) \bigg) \D s\bigg]\nonumber\\
= & \Exp_{k}^{h}\bigg[\int_{kh}^{(k+1)h}\bigg( \trace\bigl(a(X^h(s))\grad^2 f^{\epsilon}(X^h(s))\bigr) + \int_{\Act}\,b(X^h(s),\zeta)\cdot \grad f^{\epsilon}(X^h(s))v^h(s)(\D \zeta) \bigg) \D s\bigg] + \delta_k^h\,, \end{align} where $\Exp_{k}^{h}$ is the conditional expectation with respect to the $\sigma$-algebra generated by $\{X_m^{h}: m\leq k\}$ and
\begin{align*}
\delta_k^h = & \Exp_{k}^{h}\bigg[\int_{kh}^{(k+1)h}\bigg( \trace\bigl(a(\tilde{X}^h_s)\grad^2 f^{\epsilon}(\tilde{X}^h_s)\bigr) +\int_{\Act}\,b(\tilde{X}^h_s,\zeta)\cdot \grad f^{\epsilon}(\tilde{X}^h_s)v^h(s)(\D \zeta) \nonumber\\
& - \trace\bigl(a(X^h(s))\grad^2 f^{\epsilon}(X^h(s))\bigr) - \int_{\Act}\,b(X^h(s),\zeta)\cdot \grad f^{\epsilon}(X^h(s))v^h(s)(\D \zeta) \bigg) \D s\bigg] \,. 
\end{align*} Now, it is easy to see that 
\begin{equation}\label{EDiffApprox1}
\Exp\left[|\delta_k^h|^2\right] \leq \hat{C}h\int_{kh}^{(k+1)h} \Exp\left[|\tilde{X}^h_s - X^h(s)|^2 \right]\D s,    
\end{equation}
 for some positive constant $\hat{C}$, which depends on the Lipschitz constants of $a, b, \grad^2 f^{\epsilon}, \grad f^{\epsilon}$\,.

\textbf{Step-3:} Recall $n_t\df \max\{n: nh \leq t\}$\,. Since $b, \upsigma$ are bounded, for any $s\geq 0$ we have
\begin{align*}
\Exp\left( |X^h(s) - \tilde{X}^h_s|^2 \right) = \Exp\left( |X^h_{n_s} - \tilde{X}^h_s|^2 \right) \leq & 2\Exp\left(|\int_{n_sh}^{s}b(\tilde{X}_r, v^h(r))dr|^2\right) + 2\Exp\left(|\int_{n_s h}^{s}\sigma(\tilde{X}_r)dW_r|^2\right)\nonumber\\
& \leq 2(\|b\|_{\infty}^2 h +  \|\upsigma\|_{\infty}^2)h\,.
\end{align*} 
Thus, from (\ref{EDiffApprox1}) we deduce that
\begin{equation}\label{EDiffApprox1A}
\Exp\left[|\delta_k^h|\right] \leq \left(2\hat{C}h^3(\|b\|_{\infty}^2 h +  \|\upsigma\|_{\infty}^2)\right)^{\frac{1}{2}}\,.    
\end{equation} Let $\Exp_{t}^{h}$ be the conditional expectation with respect to the $\sigma$-algebra generated by $\{X^{h}(s): s\leq t\}$. Hence from (\ref{ETConvChain1AA}) and (\ref{EDiffApprox1A}), for any $0\leq s,t < \infty$ it follows that
\begin{align}\label{EDiffApprox1B} 
&\Exp_{t}^{h}\left[f^{\epsilon}(X^h(t+s))\right] - f^{\epsilon}(X^h(s)) \nonumber\\
& = \Exp_{t}^{h}\bigg[\int_{t}^{s+t}\bigg( \trace\bigl(a(X^h(r))\grad^2 f^{\epsilon}(X^h(r))\bigr) +\nonumber\\
& \int_{\Act}\,b(X^h(r),\zeta)\cdot \grad f^{\epsilon}(X^h(r))v^h(r)(\D \zeta) \bigg) \D r\bigg] + \delta^h(t, t+s)\,,   
\end{align} where $\Exp\left[|\delta^h(t, t+s)|\right] \to 0$ as $h\to 0$\,. 
Suppose $p,q$ and $t_i\leq t$ for $i\leq p$ are arbitrary and $g$, $\psi_j$ for $j\leq q$ are continuous functions (in their respective arguments) with compact support\,. Then, from (\ref{EDiffApprox1B}), we get 
\begin{align}\label{EDiffApprox1C} 
&\Exp\bigg[g\big(X^h(t_i), \int_{0}^{t_i}\int_{\Act}\psi(s,\zeta)v^h(s)(\D\zeta)\D s, i\leq p, j\leq q\big)\nonumber\\
&\bigg(f^{\epsilon}(X^h(t+s)) - f^{\epsilon}(X^h(s))  - \nonumber\\
&\int_{t}^{s+t}\bigg(\trace\bigl(a(X^h(r))\grad^2 f^{\epsilon}(X^h(r))\bigr) \nonumber\\
&+ \int_{\Act}\,b(X^h(r),\zeta)\cdot \grad f^{\epsilon}(X^h(r))v^h(r)(\D \zeta) \bigg) \D r\bigg)\bigg] \to 0 \,,
\end{align} as $h\to 0$\,. In view of (\ref{ETConvChain1AB}), by an application of the generalized dominated convergence \cite[Theorem 3.5]{serfozo1982convergence}, we obtain the following
\begin{align*} 
&\Exp\bigg[g\big(\hat{X}_{t_i}, \int_{0}^{t_i}\int_{\Act}\psi(s,\zeta)U_s(\D\zeta)\D s, i\leq p, j\leq q\big)\nonumber\\
&\bigg(f^{\epsilon}(\hat{X}_{t+s}) - f^{\epsilon}(\hat{X}_{s}) -\int_{t}^{s+t}\bigg(\trace\bigl(a(\hat{X}_{r})\grad^2 f^{\epsilon}(\hat{X}_{r})\bigr) \nonumber\\
&+ \int_{\Act}\,b(\hat{X}_{r},\zeta)\cdot \grad f^{\epsilon}(\hat{X}_{r})v^h(s)(\D \zeta) \bigg) \D r\bigg)\bigg] = 0 \,,
\end{align*} Now, using (\ref{ETConvChain1Molli}), by dominated convergence theorem, we deduce that
\begin{align}\label{EDiffApprox1D} 
&\Exp\bigg[g\big(\hat{X}_{t_i}, \int_{0}^{t_i}\int_{\Act}\psi(s,\zeta)U_s(\D\zeta)\D s, i\leq p, j\leq q\big)\nonumber\\
&\bigg(f(\hat{X}_{t+s}) - f(\hat{X}_{s})  -\int_{t}^{s+t}\bigg(\trace\bigl(a(\hat{X}_{r})\grad^2 f(\hat{X}_{r})\bigr) \nonumber\\
&+ \int_{\Act}\,b(\hat{X}_{r},\zeta)\cdot \grad f(\hat{X}_{r})v^h(s)(\D \zeta) \bigg) \D r\bigg)\bigg] = 0 \,,
\end{align} 
\textbf{Step-4:}
It is easy to see that the piecewise linear interpolations of $X_k^h$ are continuous (sample path wise) and differ from $X^{h}(\cdot)$ by $O(h)$. Hence piecewise linear interpolated process is also tight, which is going to have a (weakly) limit with continuous sample path with probability one (as the interpolated process has continuous sample path). Therefore the process $\hat{X}$ is also having continuous sample path with probability one. Thus, in view of \cref{EDiffApprox1D}, by the martingale characterization as in \cite[Theorem~2.2.9]{ABG-book}, it follows that $\hat{X}$ is a unique weak solution of \cref{E1.1} under $U\in \Uadm$\,. This completes the proof\,.
\end{proof}

A critical construction and argument in our paper is via near optimality of Lipschitz policies which is used to establish a converse inequality; which together with the above will prove to be consequential for both discounted and average cost criteria. Near optimality of Lipschitz policies were established in \cite{PYNearsmoothSCL24}, building on \cite{pradhan2022near}.

Let $v\in \Udsm$ be a Lipschitz continuous policy. Then it is easy to see that $v \in \Udsm^h$ for each $h\geq 0$\,. Let $X^{v, h}$ be the parametrized sampled Markov chain (defined as in (\ref{sampled_MDP})) corresponding to the policy $v \in \Udsm^h$ with initial condition $x\in\Rd$\,. Since $v$ is Lipschitz continuous, from (\ref{EtightA}), it is easy to see that the continuous time intetrpolated process $X^{v, h}(\cdot)$ is tight. Now, by exploiting the Lipschitz continuity of $v$ and closely following the proof steps of Theorem~\ref{TConvChain}, we have following convergence result.
\begin{theorem}\label{TConvChainA1} 
Suppose Assumptions \hyperlink{A1}{{(A1)}}--\hyperlink{A2}{{(A2)}} hold. Let $v\in \Udsm$ be a Lipschitz continuous policy\,. Let $\{X_n^{v,h}\}_{n\geq 1}$ be the parametrized sampled chain with initial condition $x_0^h$ under the piece wise constant policy $v\in \Udsm^h$. Also, let $x_0^h\to x$ as $h\to 0$\,. Then the continuous time interpolated process $X^{v,h}(\cdot)$ is tight. Let $X$ be a limit of a weakly convergent sub-sequence then $X$ is a solution of \cref{E1.1} under $v\in \Udsm$ with initial condition $x$\,. 
\end{theorem}
\section{Analysis of Discounted Cost Criterion}
In this section we study the near-optimality problem for the discounted cost criterion. 

From \cite[Theorem~3.5.6]{ABG-book}, we have the following characterization of the optimal $\alpha$-discounted cost $V_{\alpha}$\,.

\begin{theorem}\label{TD1.1}
Suppose Assumptions \hyperlink{A1}{{(A1)}}--\hyperlink{A2}{{(A2)}} hold. Then the optimal discounted cost $V_{\alpha}$ defined in \cref{OPDcost} is the unique solution in $\cC^2(\Rd)\cap\cC_b(\Rd)$ of the HJB equation
\begin{equation}\label{OptDHJB}
\min_{\zeta \in\Act}\left[\sL_{\zeta}V_{\alpha}(x) + c(x, \zeta)\right] = \alpha V_{\alpha}(x) \,,\quad \text{for all\ }\,\, x\in\Rd\,.
\end{equation}
Moreover, $v^*\in \Usm$ is $\alpha$-discounted optimal control if and only if it is a measurable minimizing selector of \cref{OptDHJB}, i.e., for a.e.\, $x\in\Rd$
\begin{align}\label{OtpHJBSelc}
&\min_{\zeta\in \Act}\left[ b(x, \zeta)\cdot \grad V_{\alpha}(x) + c(x, \zeta)\right]\nonumber\\
&= b(x,v^*(x))\cdot \grad V_{\alpha}(x) + c(x, v^*(x))\,.
\end{align}
\end{theorem}
\begin{remark}
The assumption that the running cost is Lipschitz continuous in its first argument uniformly with respect to the second is used to obtain $\cC^2(\Rd)$ solution of the HJB equation \cref{OptDHJB}. If we do not have this uniformly Lipschitz assumption, one can show that the HJB equation admits a solution in $\Sobl^{2, p}(\Rd)$, $p\geq d+1$,. 
\end{remark}
Also, from \cite[Theorem~2.1]{ABEGM-93}, we have the following complete characterization of the $\alpha$-discounted optimal policy in the space of stationary Markov strategies for the above discrete-time approximated model. 
\begin{theorem}\label{TD1.2Discre}
Suppose Assumptions \hyperlink{A1}{{(A1)}}--\hyperlink{A2}{{(A2)}} hold. Then the optimal discounted cost $\cJ_{\alpha,h}^{*}(x, c)$ defined in \cref{E2.1APM1} is a unique solution of the associated optimality equation
\begin{align}\label{DOptiEq}
&\cJ_{\alpha,h}^{*}(x, c) = \min_{\zeta \in\Act}\left[\beta\int_{\Rd}\cJ_{\alpha,h}^{*}(y, c) \cT_h(dy | x, \zeta) + c_h(x, \zeta)\right]\,,
\end{align} for all $x\in\Rd$\,. Moreover, $\cJ_{\alpha,h}^{*}(x, c)$ is a bounded, nonnegative and lower semi-continuous function\,, and $v^{h*}\in \Udsm^h$ is $\alpha$-discounted optimal control if and only if it is a measurable minimizing selector of (\ref{DOptiEq}), i.e.,
\begin{align}\label{DOptiEq1}
&\min_{\zeta \in\Act}\left[\beta\int_{\Rd}\cJ_{\alpha,h}^{*}(y, c) \cT_h(dy | x, \zeta) + c_h(x, \zeta)\right]\nonumber\\
& =\beta\int_{\Rd}\cJ_{\alpha,h}^{*}(y, c) \cT_h(dy |x, v^{h*}(x)) + c_h(x, v^{h*}(x))\,,
\end{align} for a.e.\, $x\in\Rd$\,.   
\end{theorem}
Therefore, we have there exist $v^{h*}\in \Udsm^h$ such that
\begin{equation*}
\cJ_{\alpha,h}^{v^{h*}}(x, c) = \inf_{U^h\in \Uadm^h} \cJ_{\alpha,h}^{U^h}(x, c) =  \cJ_{\alpha,h}^{*}(x, c)\,,
\end{equation*} for all $x\in \Rd$\,. The next theorem shows that the optimal value $\cJ_{\alpha,h}^{*}(x, c)$ of the approximated model converges to the optimal value $V_{\alpha}(x)$ of the diffusion model.
\begin{theorem}\label{TD1.2}
Suppose that Assumptions \hyperlink{A1}{{(A1)}}--\hyperlink{A2}{{(A2)}} hold. Then we have
\begin{equation}
\lim_{h\to 0}\cJ_{\alpha,h}^{*}(x, c)\to V_{\alpha}(x)\,.
\end{equation}
\end{theorem}
\begin{proof} 
\textbf{Step-1:}
Since $\beta = e^{-\alpha h}$, it is easy to see that 
\begin{align*}
\bigg|&\Exp_x^{U^h} \left[\sum_{k=0}^{\infty} \beta^k c_h(X_k^h, U_k^h)\right] - \Exp_x^{U^{h}} \left[\int_{0}^{\infty} e^{-\alpha s} c(X^{h}(s), v^{h}(s)) \right] \bigg| \nonumber\\
= & \bigg|\Exp_x^{U^h} \bigg[\sum_{k=0}^{\infty} \int_{kh}^{(k+1)h}\bigg(\beta^k c(X_k^h, U_k^h) - e^{-\alpha s} c(X^{h}(s), v^{h}(s)) \bigg)ds\bigg]\bigg| \nonumber\\ 
\leq &\norm{c}_{\infty}\sum_{k=0}^{\infty} \int_{kh}^{(k+1)h} \abs{\left(e^{-\alpha hk} - e^{-\alpha s} \right)} ds \nonumber\\
\leq & \norm{c}_{\infty}h\left(1 - e^{-\alpha h} \right)\sum_{k=0}^{\infty} e^{-\alpha hk} = \norm{c}_{\infty}h\,.
\end{align*}
Thus, from \cref{E2.1APM}, it is clear that
\begin{equation}\label{E2.1APMA}
\cJ_{\alpha,h}^{U^{h}}(x, c) \,=\, \Exp_x^{U^{h}} \left[\int_{0}^{\infty} e^{-\alpha s} c(X^{h}(s), v^{h}(s)) \right] + \order(h)\,.
\end{equation}
Let us define 
\begin{equation}
D^h(t) = \int_{0}^{t} e^{-\alpha s} c(X^{h*}(s), v^{h*}(s))\D  s
\end{equation}
As in Theorem~\ref{TConvChain}, viewing $v^{h*}(\cdot)$ as an element of $\cM({\infty})$, for some $U^*\in \Uadm$ we have a $v^{h*}(\cdot) \to U^*$ (under compact weak topology) as $h\to 0$\,. Also, from Theorem~\ref{TConvChain}, we have that the associated continuous-time interpolated process $X^{h*}(s)$ converges weakly to the diffusion process $X_t^*$ (solution of (\ref{E1.1}) under policy $U^*$)\,.  

Thus, by Skorohod's representation theorem, over a common probability space we have $X^{h*}(s,\omega)\to X_s^*(\omega)$ for all $s\geq 0$ and almost surely in $\omega$\,. Moreover, for any $t\geq 0$ we have
\begin{align}\label{E2.1APMA1}
&\lim_{h\to\infty}\int_{0}^{t}\int_{\Act}f(s, \zeta)v^{h*}(s,\omega)(\D \zeta)\D s = \int_{0}^{t}\int_{\Act}f(s, \zeta)U_s^*(\omega)(\D \zeta)\D s
\end{align}
for all $f\in \cC_{b}([0, t]\times \Act)$ and almost surely in $\omega$\,. Let 
\begin{equation*}
\bar{D}(t) = \int_{0}^{t} e^{-\alpha s} c(X_s^*, U^*_s)\D  s\,.
\end{equation*}
Now consider
\begin{align}
|D^h(t,\omega) - \bar{D}(t,\omega)|
\leq & |\int_{0}^{t} e^{-\alpha s}c(X^{h*}(s,\omega), v^{h*}(s,\omega))\D s - \int_{0}^{t} e^{-\alpha s}c(X_s^*(\omega), v^{h*}(s,\omega))\D s| \nonumber\\
+ & |\int_{0}^{t} e^{-\alpha s}c(X_s^*(\omega), v^{h*}(s,\omega))\D s - \int_{0}^{t} e^{-\alpha s}c(X_s^*(\omega), U_s^*(\omega))\D s|\,. 
\end{align}
Since the cost function $c$ is Lipschitz continuous in its first argument uniformly with respect to the second, it is easy to see that the first term of the above inequality converges to zero. Also, from \cref{E2.1APMA1}, we have that second term of the above inequality converges to zero. Thus, it follows that along the same sequence, for all $t\geq 0$ 
\begin{equation}\label{E2.1APMA2}
\lim_{h\to 0}D^h(t,\omega) = \bar{D}(t,\omega)\,\,\,\text{almost surely in }\,\, \omega\,.
\end{equation} Therefore, by dominated convergence theorem, for all $x\in \Rd$ we deduce that
\begin{equation*}
\lim_{h\to 0}V_{\alpha}^h(x)\to \cJ_{\alpha}^{U^*}(x, c) \,\, (\geq V_{\alpha}(x))\,.
\end{equation*} 
\textbf{Step-2:}
From the discussion as in \cite[Section 4]{PYNearsmoothSCL24}, building on \cite{pradhan2022near}, it follows that for any $\epsilon > 0$ there exists a Lipschitz continuous $\epsilon$-optimal policy $v^{\epsilon}\in \Udsm$ of (\ref{ECost1}), that is 
\begin{equation}\label{AproxCont1}
\cJ_{\alpha}^{v^{\epsilon}}(x, c) \leq V_{\alpha}(x) + \epsilon\,.
\end{equation} It is easy to see that $v^\epsilon \in \Udsm^h$ for each $h\geq 0$\,. Let $X^{\epsilon, h}$ be the parametrized sampled Markov chain (defined as in (\ref{sampled_MDP})) corresponding to the policy $v^\epsilon \in \Udsm^h$ with initial condition $x\in\Rd$\,. Now, from Theorem~\ref{TConvChainA1}, we have the continuous-time interpolated process $X^{\epsilon, h}(\cdot)$ converges weakly to the diffusion process $X^\epsilon$ corresponding to $v^\epsilon \in \Udsm$ satisfying \cref{E1.1} with initial condition $x\in\Rd$\,. Thus, by Skorohod's representation theorem, for any $t\geq 0,$ we get
\begin{align*}
\lim_{h\to 0} |&\int_{0}^{t} e^{-\alpha s}c(X^{\epsilon, h}(s,\omega), v^{\epsilon}(X^{\epsilon, h}(s,\omega)))\D s \nonumber\\
&- \int_{0}^{t} e^{-\alpha s}c(X_s^{\epsilon}(\omega), v^{\epsilon}(X_s^{\epsilon}(\omega)))\D s| =0\,, \end{align*} almost surely in $\omega$\,. This implies that
\begin{equation}\label{AproxCont2}
\lim_{h\to 0}\cJ_{\alpha, h}^{v^{\epsilon}}(x, c) = \cJ_{\alpha}^{v^{\epsilon}}(x, c)\,.
\end{equation}
Since $v^{h*}\in \Uadm_{\mathsf{sm}}^h$ is the optimal policy of the discrete-time model, it is easy to see that
\begin{equation}\label{AproxCont3}
\cJ_{\alpha, h}^{v^{h*}}(x, c) \leq \cJ_{\alpha, h}^{v^{\epsilon}}(x, c)\,.
\end{equation} Taking limit $h\to 0$, from \cref{AproxCont3}, we deduce that
\begin{equation*}
\cJ_{\alpha}^{U^{*}}(x, c) \leq \cJ_{\alpha}^{v^{\epsilon}}(x, c)\,.
\end{equation*} Thus, in view of \cref{AproxCont1}, it follows that
\begin{equation*}
V_{\alpha}(x) \leq \cJ_{\alpha}^{U^{*}}(x, c) \leq V_{\alpha}(x) + \epsilon\,.
\end{equation*} Since $\epsilon > 0$ is arbitrary, we deduce that $\cJ_{\alpha}^{U^{*}}(x, c) = V_{\alpha}(x)$ for all $x\in \Rd$\,. This completes the proof.
\end{proof}
The following theorem proves the robustness result.
\begin{theorem}\label{TD1.3}
Suppose that Assumptions \hyperlink{A1}{{(A1)}}--\hyperlink{A2}{{(A2)}} hold. Then we have
\begin{equation}
\lim_{h\to 0}\cJ_{\alpha}^{v^{h*}}(x, c) = \cJ_{\alpha}^{v^*}(x, c)\quad\text{a.e.}\,\,\, x\in \Rd\,.
\end{equation}
\end{theorem}
\begin{proof} By triangle-inequality for each $x\in\Rd$, we have
\begin{align*}
|\cJ_{\alpha}^{v^{h*}}(x, c) - \cJ_{\alpha}^{v^*}(x, c)| \leq |\cJ_{\alpha}^{v^{h*}}(x, c) - V_{\alpha}^{h}(x)| + |V_{\alpha}^{h}(x) - \cJ_{\alpha}^{v^*}(x, c)|
\end{align*} From Theorem~\ref{TD1.2} we have $|V_{\alpha}^{h}(x) - \cJ_{\alpha}^{v^*}(x, c)|\to 0$ as $h\to 0$\,. In view of \cite[Lemma~2.3.8]{ABG-book}, we have solution of (\ref{E1.1}) under $v^{h*}(\cdot)$ converges weakly to the solution of (\ref{E1.1}) under $U^{*}$. This implies that $\cJ_{\alpha}^{v^{h*}}(x, c) \to \cJ_{\alpha}^{U^{*}}(x, c)$\,. Thus, it is easy to see that $|\cJ_{\alpha}^{v^{h*}}(x, c) - V_{\alpha}^{h}(x)|\to 0$ as $h\to 0$ (since both of them converge to the same limit $\cJ_{\alpha}^{U^{*}}(x, c)$)\,. This completes the proof. 
\end{proof}
\section{Analysis of Ergodic Cost}
In this section we study the near-optimality problem for the ergodic cost criterion. The ergodic optimal control problem is studied extensively in the literature see e.g., \cite{ABG-book} and references therein. 

For this cost evolution criterion we will study the near-optimality problem under the following Lyapunov stability assumption:
\begin{itemize}
\item[\hypertarget{A3}{{(A3)}}] There exist positive constants $C_0, C_1$, a compact set $\sK$ and an inf-compact function $\Lyap\in \cC^{2}(\Rd)$ (i.e., the sub-level sets $\{\Lyap\leq k\}$ are compact or empty sets in $\Rd$\,, for each $k\in\RR$) such that for all $(x,u)\in \Rd\times \Act$, we have 
\begin{align}\label{Lyap1}
\sL_{u}\Lyap(x) \leq C_{0} I_{\{\sK\}}(x) -  C_1\Lyap(x)\,.
\end{align}
\end{itemize}
Thus, in view of (\ref{Lyap1}), from \cite[Lemma~2.5.5]{ABG-book}, it is easy to see that
\begin{equation}\label{ELyap1A}
\Exp_{x}^{U}\left[\Lyap(X_t)\right] \leq \frac{C_{0}}{C_1} + \Lyap(x) e^{- C_{1} t}
\end{equation} for all $t\geq 0$, $x\in \Rd$ and $U\in \Uadm$\,.
Now, combining \cite[Theorem~3.7.11]{ABG-book} and \cite[Theorem~3.7.12]{ABG-book}, we have the following complete characterization of the ergodic optimal control.
\begin{theorem}\label{TErgoOpt1}
Suppose that Assumptions \hyperlink{A1}{{(A1)}}--\hyperlink{A3}{{(A3)}} hold. Then the ergodic HJB equation 
\begin{equation}\label{EErgoOpt1A}
\rho = \min_{\zeta\in\Act}\left[\sL_{\zeta}V^*(x) + c(x, \zeta)\right]
\end{equation} admits a unique solution $(V^*, \rho)\in \cC^2(\Rd)\cap \sorder(\Lyap)\times \RR$ satisfying $V^*(0) = 0$\,.
Moreover, we have
\begin{itemize}
\item[(i)]$\rho = \sE^*(x_0)$
\item[(ii)] a stationary Markov control $v^*\in \Usm$ is an optimal control (i.e., $\sE(x_0, v^*) = \sE^*(x_0)$) if and only if it satisfies
\begin{align}\label{EErgoOpt1B}
&\trace\bigl(a(x)\grad^2 V^*(x)\bigr) + b(x,v^*(x))\cdot \grad V^*(x) + c(x, v^*(x))\nonumber\\ 
&\,=\,\min_{\zeta \in\Act}\left[\sL_{\zeta}V^*(x) + c(x, \zeta)\right]\,,\quad \text{a.e.}\,\, x\in\Rd\,.
\end{align}
\end{itemize} 
\end{theorem}
Also, in view of (\ref{Lyap1}) it is easy to see that, under any $v\in\Usm$ the solution of (\ref{E1.1}) is positive recurrent (see \cite[Lemma~2.5.5]{ABG-book})\,. Let $\mu_{v}$ denote the associated unique invariant measure of (\ref{E1.1}). Since the diffusion is nondegenerate, we have that the compact sets are petite. Thus, from \cite[Theorem~6.1]{meyn1993stability} (also see \cite{DMT95}), for any $f\in \cC_{BL}(\Rd)$, we have that 
\begin{equation}\label{EconvgInv1A}
\lim_{t\to\infty}|\int_{\Rd}f(y) P^t(x, v, dy) - \int_{\Rd}f(y)\mu_{v}(dy)| = 0\,,   
\end{equation} where $P^t(x, v, \cdot)$ is the transition function of the solution of (\ref{E1.1}) under $v\in \Usm$\,. 

Now, by It\^{o}-Krylov formula, from (\ref{Lyap1}) we deduce that
\begin{align*}
&\Exp_{x}^{U}\left[e^{C_1 h}\Lyap(X_{k+1}^h)| X_k^h = x\right] - \Lyap(x) \nonumber\\
=&\Exp_{x}^{U}\left[\int_0^{h} e^{C_1 s}\big(C_1 \Lyap(X_s) + \sL\Lyap(X_s)\big) ds| X_k^h = x\right]\nonumber\\
\leq & C_0 \Exp_{x}^{U}\left[\int_0^{h} e^{C_1 s} I_{\{\sK\}}(X_s)) ds| X_k^h = x\right]\,.
\end{align*} This implies
\begin{align}\label{EconvgInv1B}
\Exp_{x}^{U}&\left[\Lyap(X_{k+1}^h)| X_k^h = x\right]\nonumber\\
\leq & e^{- C_1 h}\Lyap(x) + \hat{C}_0 = \Lyap(x) - (1-e^{-C_1 h})\Lyap(x)  + \hat{C}_0
\end{align} for some positive constant $\hat{C}_0$\,. Let $\epsilon \leq 1- e^{-C_1 h}$\,. Hence from (\ref{EconvgInv1B}), we get
\begin{align*}
&\Exp_{x}^{U}\left[\Lyap(X_{k+1}^h)| X_k^h = x\right]\nonumber\\
&\leq \Lyap(x) - \epsilon \Lyap(x) - (1-\epsilon -e^{-C_1 h})\Lyap(x)  + \hat{C}_0
\end{align*}
Let $\hat{\sK}\df \{x\in\Rd: \Lyap(x) \leq \frac{\hat{C}_0}{(1-\epsilon -e^{-C_1 h})}\}$. Since $\Lyap(x)$ is an inf-compact function, the set $\hat{\sK}$ is a compact set\,. This gives us the following.
\begin{align}\label{EconvgInv1BA}
\Exp_{x}^{U}\left[\Lyap(X_{k+1}^h)| X_k^h = x\right] \leq & (1 - \epsilon) \Lyap(x) + \hat{C}_0 I_{\{\hat{\sK}\}}(x)\,.
\end{align}

Thus, from \cite[Theorem~6.3]{Meyn}, for each $h > 0$ and $f\in \cC_{BL}(\Rd)$, we also have 
\begin{equation}\label{EconvgInv1ADiscrete}
\lim_{n\to\infty} |\int_{\Rd} f(y)P^{h,n}(x, v^h, dy) - \int_{\Rd} f(y)\mu_{v}^h(dy)| = 0\,,   
\end{equation} where $P^{h,n}(x, v, \cdot)$ is the transition function of the discrete-time approximating process $X^h$ of (\ref{E1.1}) under $v\in \Usm$ and $\mu_{v}^h$ is the associated invariant measure of the process $X^h$\,.

The following theorem shows that under Lipschitz continuous policies the invariant measure of the discrete-time approximating model converges to the invariant measure of the associated continuous-time model as the parameter of the discretization approaches to zero. 
\begin{theorem}\label{TApproxInvA}
Suppose Assumptions \hyperlink{A1}{{(A1)}}--\hyperlink{A3}{{(A3)}} hold. Then for any Lipschitz continuous policy $v\in \Usm$ we have 
\begin{equation}\label{EApproxInvAB}
\lim_{h\to 0} |\int_{\Rd} f(y)\mu_{v}(dy) - \int_{\Rd} f(y)\mu_{v}^h(dy)| = 0\,,    
\end{equation} for any $f\in \cC_{BL}(\Rd)$\,.
\end{theorem}
\begin{proof} Let $f\in \cC_{BL}(\Rd)$. Now from (\ref{EconvgInv1A}) and (\ref{EconvgInv1ADiscrete}) it is easy to see that for any $\epsilon > 0$ there exists $T_1^*, T_2^*$ (respectively) large enough such that
\begin{equation}\label{EApproxInvAC}
|\int_{\Rd} f(y)P^{Nh}(x, v, dy) - \int_{\Rd} f(y)\mu_{v}(dy)| < \frac{\epsilon}{3},   
\end{equation}  for all $Nh > T_1^*$, and
\begin{equation}\label{EApproxInvAD}
|\int_{\Rd} f(y)P^{h,N}(x, v^h, dy) - \int_{\Rd} f(y)\mu_{v}^h(dy)| < \frac{\epsilon}{3}\,,    
\end{equation} for all $Nh > T_2^*$. Now
\begin{align}\label{EApproxInvAE}
&|\int_{\Rd} f(y)P^{Nh}(x, v, dy) - \int_{\Rd} f(y) P^h_N(x, v^h, dy)| \nonumber\\
&=  E\left[ |f(X_{Nh}) - f(X^h_N)|\right] \nonumber\\
& \leq \|f\|_{BL}E\left[ |X_{Nh} -X^h_N|\right] \leq \|f\|_{BL}\Exp\left[ |X_{Nh} -X^h_N|^2\right]^{\frac{1}{2}}
\end{align}
Let $v^h(\tilde{X}^h_s) \df v(\tilde{X}^h_{n})$ for $s\in [nh, (n+1) h)$\,. Since $X^h(s) = X^h_{n}$ for $s\in [nh, (n+1)h)$, we have 
\begin{align*}
&X_{nh} - X^h_n \nonumber\\
=&  X(0) - X^h_0 + \int_{0}^{nh}\left( b(X_s, v(X_s)) - b(\tilde{X}^h_s, v^h(\tilde{X}^h_s)) \right) ds \nonumber\\
&+ \int_{0}^{nh} \left( \sigma(X(s)) - \sigma(\tilde{X}_s^h\right)dW_s\,.    
\end{align*}
Thus
\begin{align}\label{EApproxInvAF}
&\|X_{nh} - X^h_n\|_2^2 \nonumber\\
\leq &  3\bigg(\|X_0 - X^h_0\|_2^2 + \|\int_{0}^{nh} \left( \upsigma(X_s) - \upsigma(\tilde{X}_s^h)\right)dW_s\|_2^2\nonumber\\
&+ \|\int_{0}^{nh}\left( b(X_s, v(X_s)) - b(\tilde{X}_s^h, v^h(\tilde{X}_s^h)) \right) ds\|_2^2\bigg)\nonumber\\
\leq &  3\bigg(\|X_0 - X^h_0\|_2^2 + \int_{0}^{nh} \Exp\left( |\upsigma(X_s) - \upsigma(\tilde{X}_s^h)|^2\right)ds\nonumber\\
&+ nh\int_{0}^{nh} \Exp\left( |b(X_s, v(X_s)) - b(\tilde{X}_s^h, v^h(\tilde{X}_s^h))|^2 \right) ds \bigg)\nonumber\\
\leq &  3\bigg(\|X_0 - X^h_0\|_2^2 + \hat{C}_2 (nh +1)\int_{0}^{nh} \Exp\left( |X_s - \tilde{X}_s^h|^2 \right) ds\bigg)\,,
\end{align} for some constant $\hat{C}_2$ which depends on the $C_0$ (see Assumption~\hyperlink{A1}{{(A1)}})\,. Suppose $s\in [(k-1)h, kh)$, then
\begin{align}\label{EApproxInvAG}
\|X_s - \tilde{X}_s^h\|_2^2 
\leq \|X_s - X_{(k-1)h}\|_2^2 + \|X_{(k-1)h} - X^h_{k-1}\|_2^2 + \|X^h_{k-1} -\tilde{X}_s^h\|_2^2\,.    
\end{align}
Since $b, \upsigma$ are bounded, we have
\begin{align}\label{EApproxInvAH}
\|X_s - X^h_{k-1}\|_2^2
\leq & 2\Exp\left(|\int_{(k-1)h}^{s}b(X_r, v(X_r))dr|^2\right) + 2\Exp\left(|\int_{(k-)h}^{s}\upsigma(X_r)dW_r|^2\right)\nonumber\\
\leq & 2(\|b\|_{\infty}h + \|\upsigma\|_{\infty})h
\end{align}
Similarly, we have $\|X^h_{k-1} -\tilde{X}_s^h\|_2^2 \leq 2(\|b\|_{\infty}h + \|\upsigma\|_{\infty})h$\,.

If $\|X(0) - X^h_0\|_2^2 \leq \hat{C}h$ for some positive constant $\hat{C}$. Let $T_3^* = \max\{T_1^*, T_2^*\} + 2$. Then for any $Nh$ such that $T_3^* \geq Nh \geq T_3^* - 1$, from (\ref{EApproxInvAF}) - (\ref{EApproxInvAH}) it follows that
\begin{align}\label{EApproxInvAI}
\|X_{nh} - X^h_n\|_2^2 
\leq & 3\hat{C}h + 12\hat{C}_2(nh + 1)(\|b\|_{\infty}h + \|\upsigma\|_{\infty})nh + \hat{C}_2(nh + 1)h\sum_{k=1}^{n} \|X_{(k-1)h} - X^h_{k-1}\|_2^2\nonumber\\
\leq & \hat{C}_4h + \hat{C}_2(T_3^* +1) h \sum_{k=1}^{n} \|X_{(k-1)h} - X^h_{k-1}\|_2^2\,,
\end{align} for some positive constant $\hat{C}_4$ (depending on $T_3^*$)\,. Let $Z_n = \max_{0\leq k\leq n}\|X_{(k-1)h} - X^h_{k-1}\|_2^2$. Thus (\ref{EApproxInvAI}) implies that
\begin{align*}
Z_n \leq \hat{C}_4 h + \hat{C}_2(T_3^* +1)h \sum_{k=1}^{n} Z_k\,.
\end{align*} Therefore, by the discrete Gr$\ddot{o}$nwall lemma, we deduce that 
\begin{align}\label{EApproxInvAI1}
Z_N \leq \hat{C}_4 h e^{\hat{C}_2(T_3^* +1)Nh} \leq \hat{C}_4 h e^{\hat{C}_2(T_3^* +1)T_3^*}
\end{align}
Now, from (\ref{EApproxInvAE}) and (\ref{EApproxInvAI1}), for some positive constant $\hat{C}_5$ it follows that
\begin{align}\label{EApproxInvAJ}
&|\int_{\Rd} f(y)P^{Nh}(x, v, dy) - \int_{\Rd} f(y) P^{h,N}(x, v^h, dy)| < \hat{C}_5\sqrt{h}
\end{align}
Hence, there exist $\delta > 0$ (small enough) such that for $|h|< \delta$, we get
\begin{align}\label{EApproxInvAK}
&|\int_{\Rd} f(y)P^{Nh}(x, v, dy) - \int_{\Rd} f(y) P^{h,N}(x, v^h, dy)| < \frac{\epsilon}{3}.
\end{align} Now combining (\ref{EApproxInvAC}), (\ref{EApproxInvAD}) and (\ref{EApproxInvAK}) we obtain the result\,.
\end{proof}
Let $L_{\Lyap}^{\infty} := \{f: \sup_{x\in \Rd}\frac{|f(x)|}{\Lyap} < \infty\}$\,. Note that under Assumption~(A2), we have (\ref{EconvgInv1B}) holds for each $h>0$\,. Thus, for the discrete-time control problem (\ref{MDP1_cost}), from \cite[Theorem~3.4]{HL-97} we have the following existence result.
\begin{theorem}\label{TErgoOptMDP1}
Suppose that Assumptions \hyperlink{A1}{{(A1)}}--\hyperlink{A3}{{(A3)}} hold. Then the average cost optimality equation 
\begin{align}\label{EErgoOptMDP1A}
\rho^h = \min_{\zeta \in\Act}\left[\int_{\Rd} \psi^{h*}(y)P^{h}(x,\zeta, dy) + c_h(x, \zeta)\right]
\end{align} admits a unique solution $(\psi^{h*}, \rho^h)\in L_{\Lyap}^{\infty}(\Rd)\times \RR$ satisfying $\psi^{h*}(0) = 0$\,.
Moreover, we have
\begin{itemize}
\item[(i)]$\rho^h = \sE_h^*(x_0)$
\item[(ii)] any stationary Markov control $v^{h*}\in \Udsm^h$ which satisfies 
\begin{align}\label{EErgoOptMDP1B}
&\min_{\zeta \in\Act}\left[\int_{\Rd} \psi^{h*}(y)P^{h}(x,\zeta, dy) + c_h(x, \zeta)\right]\nonumber\\ 
&\,=\, \left[\int_{\Rd} \psi^{h*}(y)P^{h}(x,v^{h*}(x), dy) + c_h(x, v^{h*}(x))\right]\,,
\end{align} for almost every $x\in\Rd$, is an optimal control (i.e., $\sE_h(x_0, v^{h*}) = \sE_h^*(x_0)$)\,.
\end{itemize} 
\end{theorem}
Next we show that as $h\to 0$, the optimal value $\sE_h^*(x_0)$ of the discretized model will converge to the optimal value $\sE^*(x_0)$ of the continuous-time model (\ref{E1.1})\,.
\begin{theorem}\label{TContValue1A}
Suppose that Assumptions \hyperlink{A1}{{(A1)}}--\hyperlink{A3}{{(A3)}} hold. Then we have 
\begin{align}\label{EContValue1A}
\lim_{h\to 0} \sE_h^*(x_0) = \sE^*(x_0)\,.
\end{align}    
\end{theorem}
\begin{proof}

\textbf{Step-1:}
From Theorem~\ref{TErgoOptMDP1}, we have $v^{h*}\in \Udsm^h$  such that $\sE_h^*(x_0) = \sE_h(x_0, v^{h*})$. Let $m^h\in \cM(\infty)$ be the relaxed control representation of $v^{h*}$ (i.e., $m^h(H) = \int_{\Act\times [0, \infty)} I_{\{(v^{h*}(t), t)\in H\}} m_t^h(d\zeta)dt$ for all $H\in \fB(\Act\times [0, \infty))$)\,. Thus
\begin{align}\label{EContValue1B}
\sE_h^*(x_0) &= \sE_h(x_0, v^{h*})\nonumber\\
&= \int_{\Rd}\int_{\Act} c(x,u)v^{h*}(x)(du)\mu_{v^{h*}}^{h}(dx) \nonumber\\
&= \Exp_{\mu_{v^{h*}}^{h}}^{m^{h}}\left[\int_{0}^1 c(X^{h}(s), m^h(s))ds \right]
\end{align}
Since $X^h(0)$ follows $\mu_{v^{h*}}^{h}$, it follows that $(X^h(s), m^h(s))$ is a stationary pair\,. In view of Remark~\ref{R1} and (\ref{ELyap1A}), we have that 
\begin{equation*}
\Exp_{x}^{m^h}\left[\Lyap(X_n^h)\right] \leq \frac{C_{0}}{C_1} + \Lyap(x) e^{-nh C_1}
\end{equation*} for all $n\geq 0$, $x\in \Rd$\,. Now letting $n\to \infty$, we get
\begin{equation*}
\int_{\Rd}\Lyap(y)\mu_{v^{h*}}^{h}(d y) \leq \frac{C_{0}}{C_1}\,.
\end{equation*} Since $\Lyap$ is a inf-compact function, it is easy to see that $\{\mu_{v^{h*}}^{h}\}$ is a tight sequence of invariant measures, and as a consequence of that along a sub-sequence we have $\mu_{v^{h*}}^{h} \to \mu$ (where $\mu\in \cP(\Rd)$) weakly as $h\to 0$\,. 

Let $(\bar{X}(s), \bar{m}(s))$ be a weak limit of $(X^h(s), m^h(s))$ (along a sub-sequence). It exists because of tightness of $X^h(s)$ (follows from Theorem~\ref{TConvChain}) and $M(\infty)$ is compact\,. Since $X^h(0)$ follows $\mu_{v^{h*}}^{h}$, for any $f \in \cC_{b}^{2}(\Rd)$ and $t \geq 0$ we have 
\begin{equation*}
\Exp_{\mu_{v^{h*}}^{h}}^{m^h}\left[f(X^h(t))\right] = \int_{\Rd}f(x) \mu_{v^{h*}}^{h}(dx)\,.
\end{equation*} 

Now, since the weak limit of a stationary stochastic process is also a stationary stochastic process and the initial distribution $\mu_{v^{h*}}^{h} \to \mu$ (weakly), letting $h\to 0$ from the above equation it follows that
\begin{equation*}
\Exp_{\mu}^{\bar{m}}\left[f(\bar{X}(t))\right] = \int_{\Rd}f(x) \mu(dx)\quad \text{for all}\,\,\, t\geq 0\,.
\end{equation*}This implies that $$\lim_{t\to 0} \frac{1}{t}\left(\Exp_{\mu}^{\bar{m}}\left[f(\bar{X}(t))\right] - \Exp_{\mu}^{\bar{m}}\left[f(\bar{X}(0))\right] \right) = 0$$ that is 
\begin{equation*}
\int_{\Rd}\sL_{\bar{m}}f(x)\mu(dx) = 0\,.
\end{equation*} Hence, $\mu$ is an invariant measure of the pair $(\bar{X}(s), \bar{m}(s))$ (see, \cite[Theorem 2.6.16]{ABG-book}).

Thus taking limit $h\to 0$ (along a suitable sub-sequence), from (\ref{EContValue1B}) we obtain  
\begin{align}\label{EContValue1c}
\lim_{h\to 0} \sE_h^*(x_0) &=\,\lim_{h\to 0} \Exp_{\mu_{v^{h*}}^{h}}^{m^{h}}\left[\int_{0}^1 c(X^{h}(s), m^h(s))ds \right]\nonumber\\
&=\, \Exp_{\mu}^{\bar{m}}\left[\int_{0}^1 c(\bar{X}(s), \bar{m}(s))ds \right]\nonumber\\
&=\, \limsup_{T\to\infty}\frac{1}{T}\Exp_{\mu}^{\bar{m}}\left[\int_0^T c\left(X(s),\bar{m}(s)\right)ds\right]\nonumber\\
&\geq \sE^*(x_0)\,.
\end{align}

\textbf{Step-2:}
Let $v^{\epsilon}\in \Usm$ be a Lipschitz continuous $\epsilon$-optimal policy of (\ref{ECost1}) (existence of such near optimal policy follows from the discussion as in \cite[Section 4]{PYNearsmoothSCL24})\,. Then, using Theorem~\ref{TApproxInvA}, we have the following
\begin{align}\label{EContValue1D}
& \lim_{h\to 0}\sE_h^*(x_0) \leq \lim_{h\to 0}\sE_h(x_0, v^{\epsilon}) \nonumber\\
&= \lim_{h\to 0}\Exp_{\mu_{v^{\epsilon}}^{h}}^{v^{\epsilon}}\left[\int_{0}^1 c(X^{h}(s), v^{\epsilon}(X^{h}(s)))ds\right]\nonumber\\ 
&= \Exp_{\mu_{v^{\epsilon}}}^{v^{\epsilon}}\left[\int_{0}^1 c(X(s), v^{\epsilon}(X(s)))ds\right]\nonumber\\
& = \limsup_{T\to\infty}\frac{1}{T}\Exp_{\mu_{v^{\epsilon}}}^{v^{\epsilon}}\left[\int_0^T c\left(X(s),v^{\epsilon}(X(s))\right)ds\right]\nonumber\\
& \leq \sE^*(x_0) + \epsilon\,.
\end{align}
Since $\epsilon > 0$ is arbitrary, from (\ref{EContValue1c}) and (\ref{EContValue1D}), we obtain our result\,.
\end{proof}
The following theorem shows that any optimal control of the discrete-time model (\ref{MDP1_cost}) is near optimal for (\ref{ECost1}) (for $h$ small enough)\,.
\begin{theorem}\label{TNearOpt1A}
Suppose that Assumptions \hyperlink{A1}{{(A1)}}--\hyperlink{A3}{{(A3)}} hold. Then we have 
\begin{align}\label{ENearOpt1A}
\lim_{h\to 0} \sE(x_0, v^{h*}) = \sE^*(x_0)\,.
\end{align}    
\end{theorem}
\begin{proof}
As in the proof of Theorem~\ref{TContValue1A}, let $m^h$ be the relaxed control representation of $v^{h*}$\,. Now, by triangle inequality we get
\begin{align*}
|\sE(x_0, m^{h}) - \sE^*(x_0)| \leq &|\sE(x_0, m^{h}) - \sE_h(x_0, m^{h})| \nonumber\\
& + |\sE_h(x_0, m^{h}) - \sE^*(x_0)|\,.  
\end{align*} From Theorem~\ref{TContValue1A}, it is easy to see that the second term of the above inequality converges to zero as $h\to 0$\,. In view of Remark~\ref{R1}, it is easy to see that $\sE(x_0, m^{h}), \sE_h(x_0, m^{h})$ converge to the same limit $\sE(x_0, \bar{m})$ (see (\ref{EContValue1c})) as $h\to 0$\,. This completes the proof\,.
\end{proof}
\section{Appendix}
\textbf{Another Proof of Theorem~\ref{TConvChain}:}
Since the space $\cM(\infty)$ is compact, viewing $v^{h}(\cdot)$ as an element of $\cM(\infty)$, it is easy to see that along a sub-sequence $v^{h}(\cdot)$ converges weakly to some $U\in\Uadm$. Thus, by Skorohod's representation theorem, over a common probability space we have $v^n(\cdot , \omega) \to U(\omega)$ weakly as a $\cP([0,t]\times\Act)$-valued probability measure sequence, almost surely (in $\omega$), that is
\begin{align}\label{ETConvChain1A}
& \lim_{h\to\infty}\int_{0}^{t}\int_{\Act}\hat{f}(s, \zeta)v^{h}(s,\omega)(\D \zeta)\D s \nonumber\\
&= \int_{0}^{t}\int_{\Act}\hat{f}(s, \zeta)U_s(\omega)(\D \zeta)\D s
\end{align} almost surely in $\omega$, for any $\hat{f}\in \cC_{b}([0,t]\times \Act)$ and $t \geq 0$\,.

Let $\tilde{X}^h_t$ be the solution of (\ref{E1.1}) under the control policy $v^h(\cdot)$ with initial condition $x^h_0$\,. Then for any finite stopping time $\tau$ (with respect to the filtration generated by $X^{h}(\cdot)$) and any positive constant $\delta$, we have 
\begin{align*}
&\Exp\big[|X^{h}(\tau + \delta) - X^{h}(\tau)|^2\big] \nonumber\\
&\leq  \Exp\big[|\int_{\tau}^{\tau +\delta} \left(b(\tilde{X}^h_s, v^h(\tilde{X}^h_s)) \D s + \upsigma(\tilde{X}^h_s)dW_s\right)|^2\big]\nonumber\\
&\leq 2(\|b\|_{\infty}^2 \delta +  \|\upsigma\|_{\infty}^2)\delta\,.
\end{align*} Therefore the continuous-time interpolated process $X^{h}(\cdot)$ is tight\,.

Let $f\in \cC_{BL}(\Rd)$. Thus, for any $T\geq 0$, we have
\begin{align}\label{EConvChain1} 
&\sup_{t\in [0, T]}\Exp\left[ |f(X_t) - f(X^h(t))|\right] \nonumber\\
&\leq \|f\|_{BL}\sup_{t\in [0, T]}\Exp\left[ |X_t -X^h(t)|\right] \nonumber\\
&\leq \|f\|_{BL}\sup_{t\in [0, T]}\Exp\left[ |X_t -X^{h}(t)|^2\right]^{\frac{1}{2}}
\end{align}

Let $n_t\df \max\{n: nh \leq t\}$\,. Hence, it is easy to see that $X^h(t) = X^h_{n_{t}}$. This gives us
\begin{align*}
& X_t - X^h(t) \nonumber\\
= & x - x^h_0 + \int_{0}^{n_{t}h}\left( b(X_s, U_s) - b(\tilde{X}^h_s, v^h(s)) \right) ds \nonumber\\
&+ \int_{0}^{n_{t}h} \left( \upsigma(X_s) - \upsigma(\tilde{X}^h_s)\right)dW_s \nonumber\\
&+ \int_{n_{t}h}^t\left(b(X_s, U_s) \D s + \sigma(X_s)dW_s\right)\,.    
\end{align*}
In view of this, we get the following
\begin{align}\label{EConvChain2}
&Z_T^h \df \Exp\left[ \sup_{t\in [0,T]}|X_t -X^{h}(t)|^2\right] \nonumber\\
\leq &  4\bigg(|x - x^h_0|^2 \nonumber\\
+& \Exp\left[\sup_{t\in [0,T]}\bigg(\int_{0}^{n_{t}h}\left( b(X_s, U_s) - b(\tilde{X}^h_s, v^h(s)) \right) ds \bigg)^2\right]  \nonumber\\
+& \Exp\left[\sup_{t\in [0,T]}\bigg(\int_{0}^{n_{t}h} \left( \sigma(X(s)) - \sigma(\tilde{X}^h_s)\right)dW_s\bigg)^2\right]\nonumber\\
+& \Exp\left[\sup_{t\in [0,T]}\bigg(\int_{n_{t}h}^t b(X_s, U_s) \D s + \upsigma(X_s)dW_s\bigg)^2 \right]\bigg)\nonumber\\
\leq &  4\bigg(|x - x^h_0|^2 \nonumber\\
+& 2\Exp\left[\sup_{t\in [0,T]}\bigg(\int_{0}^{n_{t}h}\left( b(X_s, U_s) - b(X_s, v^h(s)) \right) ds \bigg)^2\right] \nonumber\\
+& 2T\int_{0}^{T} \Exp\left( |b(X_s, v^h(s)) - b(\tilde{X}^h_s, v^h(s))|^2 \right) ds \nonumber\\
+ & 4\int_{0}^{n_{T}h} \Exp\left( \|\sigma(X_s) - \sigma(\tilde{X}^h_t)\|^2\right)ds \nonumber\\
+& 2\sup_{t\in [0, T]}\big(\|\int_{n_{t}h}^t b(X_s, U_s)\D s\|_2^2  + 2\|\int_{n_{t}h}^t\upsigma(X_s)dW_s\|_2^2\big)\bigg)  \nonumber\\
\leq &  4\bigg(|x - x^h_0|^2 + 2(\|b\|_{\infty}^2 h +  \|\upsigma\|_{\infty}^2)h \nonumber\\
+& 2\Exp\left[\sup_{t\in [0,T]}\bigg(\int_{0}^{n_{t}h}\left( b(X_s, U_s) - b(X_s, v^h(s)) \right) ds \bigg)^2\right] \nonumber\\
+ & 2\hat{C}_2 (T + 2)\int_{0}^{T} \Exp\left( |X_s - \tilde{X}^h_s|^2 \right) ds\bigg)\nonumber\\
\leq & 4\bigg(|x - x^h_0|^2 + 2(\|b\|_{\infty}^2 h +  \|\upsigma\|_{\infty}^2)h \nonumber\\
+&  4\Exp\left[\sup_{t\in [0,T]}\bigg(\int_{0}^{t}\left( b(X_s, U_s) - b(X_s, v^h(s)) \right) ds \bigg)^2\right] \nonumber\\
+ & 4\Exp\left[\sup_{t\in [0,T]}\bigg(\int_{n_{t}h}^{t}\left( b(X_s, U_s) - b(X_s, v^h(s)) \right) ds\bigg)^2\right] \nonumber\\
+& 2\hat{C}_2 (T + 2)\int_{0}^{T} \Exp\left( |X_s - X^h_{n_s}|^2 \right) ds \nonumber\\
&+ 2\hat{C}_2 (T + 2)\int_{0}^{T} \Exp\left( |X^h_{n_s} - \tilde{X}^h_s|^2 \right) ds\bigg)\nonumber\\
\leq & 4\bigg(|x - x^h_0|^2 + 2(5\|b\|_{\infty}^2 h +  \|\upsigma\|_{\infty}^2)h \nonumber\\
+&  4\Exp\left[\sup_{t\in [0,T]}\bigg(\int_{0}^{t}\left( b(X_s, U_s) - b(X_s, v^h(s)) \right) ds \bigg)^2\right] \nonumber\\
& + 2\hat{C}_2 (T + 2)\int_{0}^{T} \Exp\left( |X_s - X^h(s)|^2 \right) ds \nonumber\\
&+ 2\hat{C}_2 (T + 2)\int_{0}^{T} \Exp\left( |X^h_{n_s} - \tilde{X}^h_s|^2 \right) ds\bigg)\,,
\end{align} for some constant $\hat{C}_2$ which depends on the $C_0$ (see Assumption~\hyperlink{A1}{{(A1)}}))\,. 
Since $b, \upsigma$ are bounded, we have
\begin{align}\label{EConvChain3}
&\Exp\left( |X^h_{n_s} - \tilde{X}^h_s|^2 \right)\nonumber\\
\leq & 2\Exp\left(|\int_{n_sh}^{s}b(\tilde{X}_r, v^h(r))dr|^2\right) + 2\Exp\left(|\int_{n_s}^{s}\sigma(\tilde{X}_r)dW_r|^2\right)\nonumber\\
\leq & 2(\|b\|_{\infty}^2 h +  \|\upsigma\|_{\infty}^2)h\,.
\end{align}
Thus, from (\ref{EConvChain2}) and (\ref{EConvChain3}), we deduce that
\begin{align}\label{EConvChain3A}
Z_T^h \leq & 4\bigg(|x - x^h_0|^2 + 2(5\|b\|_{\infty}^2 h +  \|\upsigma\|_{\infty}^2)h \nonumber\\
& + 4\hat{C}_2 (T + 2)T(\|b\|_{\infty}^2 h +  \|\upsigma\|_{\infty}^2)h \nonumber\\ & + 4\Exp\left[\sup_{t\in [0,T]}\bigg(\int_{0}^{t}\left( b(X_s, U_s) - b(X_s, v^h(s)) \right) ds \bigg)^2\right] \nonumber\\
& + 4\hat{C}_2 (T + 1)\int_{0}^{T} Z_s^h ds\bigg)\nonumber\\
\leq & M^h(t) + 4\hat{C}_2 (T + 2)\int_{0}^{T} Z_s^hds\,,
\end{align}
where 
\begin{align*}
& M^h(t) = 4\bigg(|x - x^h_0|^2 + 2(3\|b\|_{\infty}^2 h +  \|\upsigma\|_{\infty}^2)h \nonumber\\
& + 2\hat{C}_2 (t + 1)t(\|b\|_{\infty}^2 h +  \|\upsigma\|_{\infty}^2)h \nonumber\\
& +  4\Exp\left[\sup_{t\in [0,T]}\bigg(\int_{0}^{t}\left( b(X_s, U_s) - b(X_s, v^h(s)) \right) ds \bigg)^2\right]\bigg)\,. 
\end{align*}
Since $b\in \cC_{b}(\Rd\times \Act)$, from (\ref{ETConvChain1A}), by dominated convergence theorem we have $$\Exp\left[\sup_{t\in [0,T]}\bigg(\int_{0}^{t}\left( b(X_s, U_s) - b(X_s, v^h(s)) \right) ds \bigg)^2\right] =0\,.$$ Thus, from (\ref{EConvChain3A}) we obtain
\begin{align*}
Z_T \leq M^h(T) + 8\hat{C}_2 (T + 1)\int_{0}^{T} Z_s ds\,,
\end{align*} Therefore, by the Gronwall's lemma, we deduce that 
\begin{align}\label{EConvChain5}
Z_T \leq M^h(T) + 2\hat{C}_2 (T + 1) e^{2\hat{C}_2 (T + 1)T}\int_{0}^{T}M^h(s) ds\,.
\end{align} For each $t\geq 0$, it is easy to see that $M^h(t)$ is bounded and 
$$\lim_{h\to 0} M^h(t) =0\,.$$ Thus, by dominated convergence theorem, from (\ref{EConvChain1} ) and (\ref{EConvChain5}), it follows that
\begin{align*}
\lim_{h\to 0}\sup_{t\in [0, T]}\Exp\left[ |f(X_t) - f(X^h(t))|\right] = 0\,,
\end{align*} for any $f\in \cC_{BL}(\Rd)$\,. This completes the proof\,.
\section*{Acknowledgment}
This research of S.P. is partly supported by
a Start-up Grant IISERB/R\&D/2024-25/154. S.Y. was partially supported by the Natural Sciences and Engineering Research Council of Canada (NSERC).
\bibliography{Somnath_Robustness,SerdarBibliography,Quantization}

\begin{bibdiv}
\begin{biblist}

\bib{BS86}{article}{
      author={Borkar, V.~S.},
       title={A remark on the attainable distributions of controlled
  diffusions},
        date={1986},
     journal={Stochastics},
      volume={18},
       pages={17\ndash 23},
}

\bib{BB96}{article}{
      author={Bhatt, A.~G.},
      author={Borkar, V.~S.},
       title={Occupation measures for controlled markov processes:
  characterization and optimality},
        date={1996},
     journal={Annals of Probability},
      volume={24},
       pages={1531\ndash 1562},
}

\bib{AA12}{book}{
      author={Arapostathis, A.},
       title={On the policy iteration algorithm for nondegenerate controlled
  diffusions under the ergodic criterion},
      series={In Optimization, control, and applications of stochastic
  systems},
   publisher={Birkh\"auser},
     address={Boston},
        date={2012},
}

\bib{AA13}{article}{
      author={Arapostathis, A.},
       title={On the non-uniqueness of solutions to the average cost hjb for
  controlled diffusions with near-monotone costs},
        date={2013},
     journal={arXiv preprint arXiv:1309.6307},
}

\bib{BG88I}{article}{
      author={Borkar, V.~S.},
      author={Ghosh, M.~K.},
       title={Ergodic control of multidimensional diffusions. i. the existence
  results},
        date={1988},
     journal={SIAM J. Control Optim.},
      volume={26},
       pages={112\ndash 126},
}

\bib{BG90b}{article}{
      author={Borkar, V.~S.},
      author={Ghosh, M.~K.},
       title={Ergodic control of multidimensional diffusions ii. adaptive
  control},
        date={1990b},
     journal={Appl.Math. Optim.},
      volume={21},
       pages={191\ndash 220},
}

\bib{ABG-book}{book}{
      author={Arapostathis, A.},
      author={Borkar, V.~S.},
      author={Ghosh, M.~K.},
       title={Ergodic control of diffusion processes},
      series={Encyclopedia of Mathematics and its Applications},
   publisher={Cambridge University Press},
     address={Cambridge},
        date={2012},
      volume={143},
      review={\MR{2884272}},
}

\bib{BR-02}{article}{
      author={Barles, G.},
      author={Jakobsen, E.~R.},
       title={On the convergence rate of approximation schemes for
  {Hamilton-Jacobi-Bellman} equations},
        date={2002},
     journal={ESAIM: Mathematical Modelling and Numerical Analysis -
  Mod\'elisation Math\'ematique et Analyse Num\'erique},
      volume={36},
      number={1},
       pages={33\ndash 54},
         url={http://www.numdam.org/articles/10.1051/m2an:2002002/},
      review={\MR{1916291}},
}

\bib{BJ-06}{article}{
      author={Barles, G.},
      author={Jakobsen, E.~R.},
       title={Error bounds for monotone approximation schemes for
  hamilton-jacobi-bellman equations},
        date={2006},
     journal={SIAM Journal on Numerical Analysis},
      volume={43},
      number={2},
       pages={540\ndash 558},
}

\bib{JPR-19P}{article}{
      author={Jakobsen, E.~R.},
      author={Picarelli, A.},
      author={Reisinger, C.},
       title={Improved order 1/4 convergence for piecewise constant policy
  approximation of stochastic control problems},
        date={2019},
     journal={Electronic Communications in Probability},
      volume={24},
       pages={1\ndash 10},
}

\bib{kushner1990numerical}{article}{
      author={Kushner, H.~J.},
       title={Numerical methods for stochastic control problems in continuous
  time},
        date={1990},
     journal={SIAM Journal on Control and Optimization},
      volume={28},
      number={5},
       pages={999\ndash 1048},
}

\bib{kushner2001numerical}{book}{
      author={Kushner, H.~J.},
      author={Dupuis, P.~G.},
       title={Numerical methods for stochastic control problems in continuous
  time},
   publisher={Springer Science \& Business Media},
        date={2001},
      volume={24},
}

\bib{kushner2012weak}{book}{
      author={Kushner, H.J.},
       title={Weak convergence methods and singularly perturbed stochastic
  control and filtering problems},
   publisher={Springer Science \& Business Media},
        date={2012},
}

\bib{fleming2006controlled}{book}{
      author={Fleming, W.~H.},
      author={Soner, H.~M.},
       title={Controlled markov processes and viscosity solutions},
   publisher={Springer Science \& Business Media},
        date={2006},
      volume={25},
}

\bib{KD92}{book}{
      author={Kushner, H.~J.},
      author={Dupuis, P.~G.},
       title={Numerical methods for stochastic control problems in continuous
  time},
   publisher={Springer-Verlag, New York,},
        date={1992},
}

\bib{KH77}{book}{
      author={Kushner, H.~J.},
       title={Probability methods for approximations in stochastic control and
  for elliptic equations},
      series={Math. Sci. Eng.},
   publisher={Academic Press, New York},
        date={1977},
      volume={129},
}

\bib{KH01}{book}{
      author={Kushner, H.~J.},
       title={Heavy traffic analysis of controlled queueing and communication
  networks},
      series={Stoch. Model. Appl. Probab.},
   publisher={Springer-Verlag, New York},
        date={2001},
      volume={47},
}

\bib{KN98A}{article}{
      author={Krylov, N.},
       title={On the rate of convergence of finite-difference approximations
  for bellman’s equations},
        date={1998},
     journal={St. Petersburg Math. J.},
      volume={9},
       pages={639\ndash 650},
}

\bib{KN2000A}{article}{
      author={Krylov, N.},
       title={On the rate of convergence of finite-difference approximations
  for bellmans equations with variable coefficients},
        date={2000},
     journal={Probab Theory Relat Fields},
      volume={117},
       pages={1\ndash 16},
         url={https://doi.org/10.1007/s004400050264},
}

\bib{kushner2014partial}{article}{
      author={Kushner, H.~J.},
       title={A partial history of the early development of continuous-time
  nonlinear stochastic systems theory},
        date={2014},
     journal={Automatica},
      volume={50},
      number={2},
       pages={303\ndash 334},
}

\bib{pradhanyuksel2023DTApprx}{article}{
      author={Pradhan, Somnath},
      author={Y{\"u}ksel, Serdar},
       title={Controlled diffusions under full, partial and decentralized
  information: Existence of optimal policies and discrete-time approximations},
        date={2023},
     journal={arXiv preprint arXiv:2311.03254},
}

\bib{KN2001AA}{article}{
      author={Krylov, N.},
       title={Mean value theorems for stochastic integrals},
        date={2001},
     journal={Ann. Probab.},
      volume={29},
       pages={385\ndash 410},
         url={https://doi.org/10.1214/aop/1008956335},
}

\bib{KN99AA}{article}{
      author={Krylov, N.},
       title={Approximating value functions for controlled degenerate diffusion
  processes by using piece-wise constant policies},
        date={1999},
     journal={Electron. J. Probab.},
      volume={4},
       pages={1\ndash 19},
         url={https://doi.org/10.1214/EJP.v4-39},
}

\bib{jakobsen2019improved}{article}{
      author={Jakobsen, E.~R.},
      author={Picarelli, A.},
      author={Reisinger, C.},
       title={Improved order 1/4 convergence for piecewise constant policy
  approximation of stochastic control problems},
        date={2019},
}

\bib{KBHIFAC23}{article}{
      author={Karumanchi, Sambhu~H.},
      author={Belabbas, Mohamed~A.},
      author={Hovakimyan, Naira},
       title={Empirical dynamic programming for controlled diffusion
  processes},
        date={2023},
        ISSN={2405-8963},
     journal={IFAC-PapersOnLine},
      volume={56},
      number={2},
       pages={11235\ndash 11241},
        note={22nd IFAC World Congress},
}

\bib{PYNearsmoothSCL24}{article}{
      author={Pradhan, Somnath},
      author={Yüksel, Serdar},
       title={Near optimality of lipschitz and smooth policies in controlled
  diffusions},
        date={2024},
        ISSN={0167-6911},
     journal={Systems \& Control Letters},
      volume={193},
       pages={105943},
  url={https://www.sciencedirect.com/science/article/pii/S0167691124002317},
}

\bib{pradhan2022near}{article}{
      author={Pradhan, S.},
      author={Y{\"u}ksel, S.},
       title={Continuity of cost in {B}orkar control topology and implications
  on discrete space and time approximations for controlled diffusions under
  several criteria},
        date={2024},
     journal={Electronic Journal of Probability},
      volume={29},
       pages={1\ndash 32},
}

\bib{bayraktar2022approximate}{article}{
      author={Bayraktar, E.},
      author={Kara, A.~D.},
       title={Approximate q learning for controlled diffusion processes and its
  near optimality},
        date={2023},
     journal={SIAM Journal on Mathematics of Data Science},
      volume={5},
      number={3},
       pages={615\ndash 638},
}

\bib{KSYContQLearning}{article}{
      author={Kara, A.D},
      author={Saldi, N.},
      author={Y\"uksel, S.},
       title={Q-learning for {M}{D}{P}s with general spaces: Convergence and
  near optimality via quantization under weak continuity},
        date={2023},
     journal={Journal of Machine Learning Research},
       pages={1\ndash 34},
}

\bib{Adams}{book}{
      author={Adams, R.~A.},
       title={Sobolev spaces},
   publisher={Academic Press, New York},
        date={1975},
}

\bib{BG90}{article}{
      author={Borkar, Vivek~S.},
      author={Ghosh, Mrinal~K.},
       title={Controlled diffusions with constraints},
        date={1990},
     journal={Journal of Mathematical Analysis and Applications},
      volume={152},
      number={1},
       pages={88\ndash 108},
         url={https://doi.org/10.1016/0022-247X(90)90094-V},
}

\bib{PBill-book}{book}{
      author={Billingsley, Patrick},
       title={Convergence of probability measures},
      series={2nd ed.},
   publisher={Wiley},
     address={New York},
        date={1999},
}

\bib{Bor-book}{book}{
      author={Borkar, Vivek~S.},
       title={Optimal control of diffusion processes},
      series={Pitman Research Notes in Mathematics Series},
   publisher={Longman Scientific \& Technical, Harlow; copublished in the
  United States with John Wiley \& Sons, Inc., New York},
        date={1989},
      volume={203},
        ISBN={0-582-03540-6},
      review={\MR{1005532}},
}

\bib{serfozo1982convergence}{article}{
      author={Serfozo, R.},
       title={Convergence of {L}ebesgue integrals with varying measures},
        date={1982},
     journal={Sankhy{\=a}: The Indian Journal of Statistics, Series A},
       pages={380\ndash 402},
}

\bib{ABEGM-93}{article}{
      author={Arapostathis, Ari},
      author={Borkar, Vivek~S.},
      author={Fernández-Gaucherand, Emmanuel},
      author={Ghosh, Mrinal~K.},
      author={Marcus, Steven~I.},
       title={Discrete-time controlled markov processes with average cost
  criterion: a survey},
        date={1993},
     journal={SIAM J. Control Optim.},
      volume={31},
      number={2},
       pages={282\ndash 344},
}

\bib{meyn1993stability}{article}{
      author={Meyn, S.~P.},
      author={Tweedie, R.~L.},
       title={Stability of {M}arkovian processes iii: Foster-lyapunov criteria
  for continuous-time processes},
        date={1993},
     journal={Advances in Applied Probability},
       pages={518\ndash 548},
}

\bib{DMT95}{article}{
      author={Down, D.},
      author={Meyn, S.~P.},
      author={Tweedie, R.~L.},
       title={{Exponential and Uniform Ergodicity of Markov Processes}},
        date={1995},
     journal={The Annals of Probability},
      volume={23},
      number={4},
       pages={1671 \ndash  1691},
         url={https://doi.org/10.1214/aop/1176987798},
}

\bib{Meyn}{article}{
      author={Meyn, S.~P.},
      author={Tweedie, R.},
       title={Stability of {M}arkovian processes {I}: Criteria for
  discrete-time chains},
        date={1992},
     journal={Advances in Applied Probability},
      volume={24},
       pages={542\ndash 574},
}

\bib{HL-97}{article}{
      author={Hern\'andez-Lerma, O.},
      author={Lasserre, J.B.},
       title={Policy iteration for average cost markov control processes on
  borel spaces},
        date={1997},
     journal={Acta Applicandae Mathematicae},
      volume={47},
       pages={125\ndash 154},
         url={https://doi.org/10.1023/A:1005781013253},
}

\end{biblist}
\end{bibdiv}

\end{document}